\documentclass{amsart}

\usepackage[latin1]{inputenc}
\usepackage{amsfonts}
\usepackage{amsmath}
\usepackage{amsthm}
\usepackage{amssymb}
\usepackage{latexsym}
\usepackage{enumerate}

\newtheorem{theorem}{Theorem}[section]
\newtheorem{lemma}[theorem]{Lemma}
\newtheorem{proposition}[theorem]{Proposition}
\newtheorem{corollary}[theorem]{Corollary}

\newtheorem{definition}[theorem]{Definition}

\numberwithin{equation}{section}

\begin{document}

\newcommand{\cc}{\mathfrak{c}}
\newcommand{\N}{\mathbb{N}}
\newcommand{\Q}{\mathbb{Q}}
\newcommand{\R}{\mathbb{R}}
\newcommand{\PP}{\mathbb{P}}
\newcommand{\I}{\mathbb{I}}
\newcommand{\forces}{\Vdash}

\title{On large indecomposable Banach spaces}

\author{Piotr Koszmider}
\thanks{The  author was partially supported by Polish Ministry of
Science and Higher Education research grant N N201
386234. } 
\address{Institute of Mathematics, Polish Academy of Sciences,
ul. \'Sniadeckich 8,  00-956 Warszawa, Poland}

\email{\texttt{P.Koszmider@impan.pl}}

\subjclass{}
\date{}
\keywords{}

\begin{abstract}
Hereditarily indecomposable Banach spaces may have density at most
continuum (Plichko-Yost, Argyros-Tolias). In this paper we show that this cannot be proved for
indecomposable Banach spaces. We provide the first  example of an indecomposable
Banach space of density $2^{2^\omega}$. The space exists consistently, 
 is of the form $C(K)$ and
it has few operators in the sense that any  bounded linear operator $T:C(K)\rightarrow C(K)$
satisfies $T(f)=gf+S(f)$ for every $f\in C(K)$, where $g\in C(K)$ and $S:C(K)\rightarrow C(K)$ is weakly
compact (strictly singular). 
\end{abstract}

\maketitle

\markright{}

\section{Introduction}

We say that an infinite dimensional Banach space $X$ is indecomposable (I) if 
whenever $X=A\oplus B$, then either $A$ or $B$ is finite dimensional.
First indecomposable Banach spaces constructed by Gowers and Maurey were hereditarily indecomposable (HI)
(see \cite{gowersmaurey}, \cite{linden}).
That is, all their infinite dimensional subspaces were indecomposable. 
These spaces were separable, but many nonseparable constructions followed e.g., \cite{tolias}.
However 
on every HI space there is an injective operator
into $l_\infty$, and so HI spaces have density at most $2^\omega$  (see \cite{plickoyost}).

Different indecomposable Banach spaces were constructed in \cite{few} and in several papers that followed
like \cite{plebanek}, \cite{rogerio}, \cite{iryna}. These  spaces are of the form $C(K)$ and 
so they have many decomposable subspaces, actually they must
 have all separable Banach spaces as subspaces, and cannot be separable. Most of these constructions
are of density $2^\omega$ with the exception of \cite{rogerio} which has consistently a smaller density.

In this paper we show that  indecomposable Banach spaces
of the form $C(K)$  may have density bigger than $2^\omega$
which provides the first example of such Banach spaces. Our construction is not absolute,
that is, we prove that it is consistent that such spaces exist. 
Constructions of Banach spaces which essentially depend on additional set-theoretic axioms or methods
are not uncommon among nonseparable spaces. Probably the most known examples are of Shelah from \cite{shelah2} and
the so called Kunen space (see \cite{negrepontis}) where despite its nonseparability we
do not have uncountable biorthogonal systems. Newer important examples include generic Banach spaces of
Lopez-Abad and Todorcevic (\cite{stevogeneric}) or Banach spaces without support sets of \cite{rolewicz}.
The necessity of these additional combinatorial methods  in some of the above mentioned results
is shown in \cite{stevobio}.

The main result of this paper is:

\begin{theorem}\label{maintheorem} It is consistent with $2^\omega=\omega_1$,
$2^{\omega_1}=\omega_2$  that there is  a compact Hausdorff space $L$ such that:
\begin{enumerate}
\item The density of $C(L)$ is $2^{2^\omega}=\omega_2>\omega_1=2^\omega$,
\item Every linear bounded operator $T: C(L)\rightarrow C(L)$ satisfies
$T(f)=gf+S(f)$ for every $f\in C(L)$ where $g\in C(L)$ and $S$ is weakly compact (strictly singular),
\item $C(L)$ is an indecomposable Banach space, in particular it has no
infinite dimensional complemented subspaces of density less or equal to  $2^\omega$,
\item $C(L)$ is not isomorphic to any of its proper subspaces nor any
of its proper quotients, in particular it is not isomorphic to its hyperplanes.
\end{enumerate}
\end{theorem}

This is an immediate consequence of \ref{finaltheorem}.
We do not know if the existence
of such spaces may be proved in ZFC or if this an undecidable problem.

Many other basic related questions remain open. For example, whether there is any bound of the densities of indecomposable Banach spaces (a question due to S. Argyros), however 
one can show that there is a bound on the densities of
Banach spaces with the properties that we obtain.

Our construction is related to the construction from \cite{big} of a $C(K)$ satisfying the first two items of the above in
\ref{maintheorem}.  The space $K$ of \cite{big} is totally disconnected and so has nontrivial projections all of which
can be characterized as being  strictly singular perturbations of  multiplications by a characteristic functions
of  clopen sets. The modification however is very indirect,  as standard amalgamations of approximating subspaces
from \cite{big} do not preserve the connectedness. Actually quite complicated partial order
which is used here to force the compact $L$ is designed to take care of the connectedness
of the amalgamation. In a sense, what is required is that the metrizable approximations to the final $L$ must
predict all possible future amalgamations. So, much more complicated structure of
these approximations is required in the present paper than in \cite{big}.

 It would be very interesting to have a
more direct argument which would give the connectedness of the amalgamation.
We also feel that the right language for such constructions should be that of Banach algebras,
that is, one should rather construct a $C(L)$ than $L$. Then the connectedness of $L$ is, of course, equivalent
the nonexistence of nontrivial idempotent elements in the algebra $C(L)$ which in this context
(of few operators) gives the nonexistence of nontrivial (in more relaxed sense of finite dimensional perturbation) idempotents
in the algebra $\mathcal B(C(L))$ of operators on $C(K)$.  However the question of the existence of 
nontrivial idempotents in an algebra given by its generators could be a difficult question 
 (e.g., \cite{voiculescu}, \cite{idempotents}) which requires a complicated machinery.

 Our approach has  this flavour of considering the algebra $C(L)$ only partially.
Namely, we take an order complete Banach algebra $C(K)$ for $K$ extremely disconnected
and carefully choose a subfamily $\mathcal F\subseteq C(K)$ consisting of
functions whose ranges are included in $[0,1]$ which generates
the  $C(L)$ that is
$L=(\Pi\mathcal F)[K]$ where $\Pi\mathcal F: K\rightarrow [0,1]^{\mathcal F}$
is given by $(\Pi\mathcal F)(x)(f)=f(x)$ for any $f\in\mathcal F$. 
However most of our arguments are done on the level of underlying compact spaces
and are, at first sight, quite detached from the Banach space theoretic structure of the
induced function spaces.

Section 1 mainly deals with the extremely disconnected compact K and 
the spaces of the form $(\Pi\mathcal F)[K]$ for  $\mathcal F\subseteq C(K)$.
In Section 2 we analyze an auxiliary partial order of some countable 
subsets of $C(K)$ which could approximate our final $\mathcal F$ which would 
determine $L$ as above. Though the final space $L$ has no nontrivial continuous
mappings into itself, the approximating spaces  have many homeomorphism
which help us in amalgamations of the  approximations into better approximations.
Section 4 is devoted to the analysis of the final partial order of approximations
which incorporates the auxiliary one from the previous section.
In the next Section 5 we consider adding suprema of sequences of functions
to our approximating families $\mathcal F\subseteq C(K)$. 
This section contains results that generalize the techniques of Section 4 of \cite{few}.
Section 6 is devoted to the main extension lemma, where an approximation
to the final $\mathcal F\subseteq C(K)$ can be enriched by a supremum of
a sequence of pairwise disjoint functions. The last two sections 7 and 8
deal with the final realization of the construction. Only in these two last sections
some knowledge of consistency proofs by forcing is required from the reader.

It is worthy to mention that our $L$ is strongly rigid compact space, that is its only
continuous mappings are constant functions or the identity (see 5.4. of \cite{centripetal}).
A first compact strongly rigid space was obtained by H. Cook (\cite{cook}) and V. Trnkova
proved that there are no bounds on the weight nor size of such spaces (\cite{trnkova}) in ZFC.
However,  $C(K)$s for the compact spaces of \cite{cook} or \cite{trnkova}
are decomposable. On the other hand
there is a necessary and sufficient condition for $C(K)$ to have few operators as above
in terms of some rigidity of the dual ball of $C(K)$ (see Theorem 23 of  \cite{fewsur}).
As in the case of the construction from \cite{big} we conjecture that the space
of this paper can be obtained directly from a combinatorial
principle called Velleman's simplified morass (\cite{velleman}) instead of
a forcing argument.

\noindent We would like to thank Rog\'erio Fajardo for countless hours of discussions 
back in the years 2005-07
concerning the problem of large indecomposable spaces, to Roman Pol for his remarks related to strongly rigid compact spaces
and to Adam Skalski for his remarks on idempotent elements in $C^*$-algebras.

\noindent Most unexplained set-theoretic and logical concepts can be found
in \cite{kunen} or \cite{jech}, topological terminology is based on
\cite{engelking} and the basics on $C(K)$ Banach spaces can be
found in \cite{semadeni}.

\section{Continuous images of some extremely disconnected compact space}

 Let  $A, A'\subseteq \omega_2$. We say that $A<A'$ if each element
of $A$ is smaller in the ordinal order on $\omega_2$ than all elements of $A'$.
Consider $Fr(A)$, the free Boolean algebra generated by 
independent family $(a_\xi:\xi\in A)$ and its completion $Co(A)$.
It is well known that $Fr(A)$ and so $Co(A)$ satisfy the c.c.c.,
so every element of $Co(A)$ can be considered as a supremum of an
antichain from $Fr(A)$. 

If $A\subseteq B$, then there is a natural embedding $i_{BA}: Co(A)\rightarrow Co(B)$
which induces, by the Stone duality a continuous surjection $\rho_{AB}: K_B
\rightarrow K_A$, where $K_A$ and $K_B$ are the Stone spaces of $Co(A)$ and
$Co(B)$ respectively.  We put $K=K_{\omega_2}$. 
If $L$ is a compact Hausdorff space, then $C(L)$ denotes
the Banach space of all real-valued continuous functions on $L$
with the supremum norm.
$C_I(K)$ will denote the set of all elements of $C(K)$ whose ranges are included
in $[0,1]$.

In general 
if $\mathcal A$ is a Boolean algebra, then $S(\mathcal A)$ denotes
its Stone space,  $[a]$ will denote the basic clopen set of the Stone space
of a Boolean algebra where $a$ belongs. So, in particular
$K_A=S(Co(A))$. We will use the notation $S(Co(A))$ when we want to exploit the
fact that its points are subsets of $Co(A)$ (ultrafilters).

Suppose $\alpha\in\omega_2$. Of course the Stone
space of $Fr([\alpha,\alpha+\omega))$  is homeomorphic with
the Cantor set $2^\omega$. By the standard surjection from $2^\omega$ onto
$[0,1]$  we mean the function $f(x)=\Sigma_{n\in\N}{x(i)\over 2^i}$ which is irreducible.
$d_\alpha: K\rightarrow [0,1]$ will denote the function
obtained in similar way from $S(Fr([\alpha,\alpha+\omega]))$ instead of $2^\omega$:

\begin{definition} Suppose $\alpha\in\omega_2$.
$d_\alpha: K\rightarrow [0,1]$ is defined by
$$d_\alpha(x)=\sum_{n\in\N}{t_x(i)\over 2^i},$$
where $t_x(i)=1$ if $a_{\alpha+i}\in x$ and $t_x(i)=0$ if $a_{\alpha+i}\not\in x$.
\end{definition}

\begin{definition}\label{amalgamationalgebras}
Suppose that $A, B\subseteq \omega_2$.
$Co(A)\otimes Co(B)$ is the subalgebra of $Co(A\cup B)$ generated
by $Co(A)\cup Co(B)$. 
\end{definition}

\begin{lemma}\label{ultrafilters1} Let $A, B\subseteq \omega_2$. Let $u\in S(Co(A))$ and $v\in S(Co(B))$
be such that $u\cap Co(A\cap B)=v\cap Co(A\cap B)$. Then there is a unique ultrafilter
$w$ in the Stone space of $Co(A)\otimes Co(B)$  such that $u,v\subseteq w$.
\end{lemma}

\begin{proof}
If $u\cup v$ does not have the finite intersection property, there are
disjoint $\Sigma (a_i\cap b_i)\in u$ and $\Sigma (a_i'\cap c_i)\in v$ such that
$a_i, a_i'\in Fr(A\cap B)$, $b_i\in Fr(A\setminus B)$ and
$c_i\in Fr(B\setminus A)$. Hence $\Sigma a_i\in u\cap Co(A\cap B)$ 
and $\Sigma a_i'\in v\cap Co(A\cap B)$ as they are bigger elements.
The hypothesis then, gives that $a_i\cap a_j'\not=0$ for
some $i, j$.  But then $a_i\cap a_j'\cap b_i\cap c_j\not=0$ by
the independence of all elements of
$Fr(A\cap B)$, $Fr(A\setminus B)$, $Fr(B\setminus A)$ which contradicts
the disjointness of the original elements.
\end{proof}

\begin{lemma}\label{ultrafilters} Let $A, B\subseteq \omega_2$. Let $u\in S(Co(A))$ and $v\in S(Co(B))$
be such that $u\cap Co(A\cap B)=v\cap Co(A\cap B)$. Then there are ultrafilters
$w\in K_{A\cup B}$ and $w'\in K$ such that $u,v\subseteq w, w'$.
\end{lemma}

\begin{proof}
Extend the ultrafilter from \ref{ultrafilters1}.
\end{proof}

\begin{corollary}\label{disjointultrafilters}
Suppose $A, B\subseteq \omega_2$ are disjoint and 
$u\in K_A$ and $v\in K_B$. Then there are ultrafilters $w\in K_{A\cup B}$
and  $w'\in K$ such that $u,v\subseteq w, w'$.
\end{corollary}

\begin{definition}\label{definitionpinabla}
If $F\subseteq C_I(K)$, then by $\Pi F: K\rightarrow [0,1]^F$ we mean 
the function defined by
$$(\Pi F(x))(f)=f(x),$$
for each $x\in K$ and each $f\in F$.
The image of $\Pi F$ in
$[0,1]^F$ is denoted by $\nabla F$.
If $\mathcal G\subseteq\mathcal F$, then $\pi_{\mathcal G \mathcal F}$ denotes
the projection from $\nabla\mathcal F$ onto $\nabla \mathcal G$.
\end{definition}

\begin{lemma}\label{subproductsconnected}
Suppose that $\mathcal F\subseteq C_I(K)$. $\nabla \mathcal F$ is
connected if and only if $\nabla\{f_1,...f_m\}$ is connected for any finite
$f_1,...f_m\in \mathcal F$.
\end{lemma}
\begin{proof}
$\nabla\{f_1, ..., f_m\}$ is a continuous image
of $\nabla \mathcal F$ by taking the canonical projections onto
the coordinates indexed by $f_1, ..., f_m$ hence the connectedness is preserved.

Now suppose $A, B\subseteq  \nabla \mathcal F$ are nonempty complementary clopen 
sets. Any open set can be covered by basic open sets determined by finitely
many coordinates. But $A,B$ are clopen, hence compact, so such covers have finite
subcovers. It follows that 
$$A=\bigcup_{i\leq n}\bigcap_{j\leq k} (\prod\mathcal F) [f^{-1}_{i,j}[U_{i,j}]],
\ B=\bigcup_{i'\leq n'}\bigcap_{j'\leq k'}  (\prod\mathcal F) [g^{-1}_{i',j'}[V_{i',j'}]],$$
where $\{f_{i,j}, g_{i',j'}: {i\leq n}, {j\leq k}, {i'\leq n'}, {j'\leq k'}\}\subseteq \mathcal F$
and $U_{i,j}, V_{i',j'}$ are open subsets of $[0,1]$.
However this means that already $\nabla\{f_{i,j}, g_{i',j'}: {i\leq n}, {j\leq k}, {i'\leq n'}, {j'\leq k'}\}$
is not connected.
\end{proof}

\begin{lemma}\label{nablaofdealphas} Suppose $A'\subseteq \omega_2$ consists
only of limit ordinals. 
 Then
$\nabla\{d_\alpha: \alpha\in A'\}$ is homeomorphic to $[0,1]^{A'}$.
\end{lemma}
\begin{proof} As the sets $[\alpha, \alpha+\omega)$ for $\alpha\in A'$ are pairwise
disjoint, using \ref{ultrafilters}, for any $(x_\alpha)_{\alpha\in A'}\in \Pi_{\alpha\in A'} S(Fr([\alpha, \alpha+\omega))$
there is $x\in S(Fr(\bigcup_{\alpha\in A'}[\alpha,\alpha+\omega))$ such that $x\cap Fr([\alpha,\alpha+\omega))=x_\alpha$.
So  $\Pi\{d_\alpha: \alpha\in A'\}$ is onto $[0,1]^{\{d_\alpha: \alpha\in A'\}}$
which is of course homeomorphic to $[0,1]^{A'}$,
hence its image  $\nabla\{d_\alpha: \alpha\in A'\}$ is homeomorphic to $[0,1]^{A'}$
as required.
\end{proof}

\begin{definition}
$f:K\rightarrow \R$ is said to depend on a Boolean algebra $\mathcal A\subseteq Co(\omega_2)$ if and 
only if whenever $x\cap \mathcal A=y\cap \mathcal B$, then $f(x)=f(y)$.
If 
 $A\subseteq \omega_2$ we say that $f$ depends on $A$ if it depends on $Co(A)$.
\end{definition}

\begin{lemma}\label{countabledependence}
Each $f\in C(K)$ depends on a countable set $A\subseteq \omega_2$.
\end{lemma}
\begin{proof}
For any rationals $a<b<c<d$ consider a closed $F_{b,c}=f^{-1}[[b,c]]$ 
and an open $U_{a,d}=f^{-1}[(a,d)]$. Clearly $F_{b,c}\subseteq U_{a,d}$.
Let ${\mathcal A}\subseteq Co(\omega_2)$ be a countable subalgebra such that
for any rationals $a<b<c<d$ there is a clopen $U\in \mathcal A$
such that
 $$F_{a,b}\subseteq U \subseteq U_{a,b}.$$
It is enough to show that if $u\cap{\mathcal A}=v\cap{\mathcal A}$, then
$f(v)=f(u)$.  If not, then there are rational $a<b<c<d$ such that
$u\in F_{b,c}$ and $v\not\in U_{a,d}$, hence there is $U\in {\mathcal A}$
such that $u\in U$ and $v\not\in U$.
Now let $A\subseteq \omega_2$ be a countable set such that $\mathcal A\subseteq Co(A)$,
it exists since $Co(\omega_2)$ is c.c.c.
\end{proof}

\begin{lemma} Suppose that all functions from $\mathcal F\subseteq C_I(K)$
depend on $A\subseteq\omega_2$ and that $ A\cap [\alpha+\omega)=\emptyset$.
Then $\nabla(\mathcal F\cup\{d_\alpha\})=(\nabla \mathcal F)\times [0,1]$.
\end{lemma}
\begin{proof} Let $t\in \nabla\mathcal F$ and $x\in [0,1]$
and let $u', v'\in Co(\omega_2)$ be such that
$(\Pi\mathcal F)(u')=t$ and $d_\alpha(v')=x$.
Put $v=v'\cap Co([\alpha,\alpha+\omega))$ and $u=u'\cap Co(A)$.
Use \ref{disjointultrafilters} to obtain $w\in Co(\kappa)$ with
$w\cap Co([\alpha,\alpha+\omega))=v'\cap Co([\alpha,\alpha+\omega))$ and
 $w\cap Co(A)=u'\cap Co(A)$. We have that $(\Pi\mathcal F)(w)=t$ and
$d_\alpha(w)=x$ and so $\Pi(\mathcal F\cup\{d_\alpha\})(w)=(t,x)$.

\end{proof}

\begin{lemma}\label{cardinality}
If $A\subseteq\omega_2$ is countable then $C(K_A)$ has the 
cardinality of the continuum. In particular the cardinality of 
the family of all functions which depend on a countable $A$ is 
continuum.
\end{lemma}
\begin{proof} $Fr(A)$ is dense in $Co(A)$.
As $Fr(A)$ has c.c.c., any element of $Co(A)$ is 
the union of a countable antichain in $Fr(A)$. Hence,
the cardinality of $Co(A)$ is continuum.
By the Stone-Weierstrass theorem any element of
$C(K_A)$ can be approximated by a finite  linear combination of
characteristic functions of elements from $Co(K)$, so elements
of $C(K_A)$ can be associated with countable subsets of 
$[\R\times Co(A)]^{<\omega}$ which has the size of continuum
as required.
\end{proof}

\begin{proposition}\label{generalbijections}
Let  $\sigma:\omega_2
\rightarrow\omega_2$ be a bijection.
Then there is a unique isomorphism  denoted $h(\sigma)$
of the Boolean algebras $co(Fr(\omega_2))$ and
$co(Fr(\omega_2))$ which sends $a_\xi$ to
$a_{\sigma(\xi)}$. We say that such an
isomorphism is induced by $\sigma$.
Then there is a unique homeomorphism 
of   $K$ denoted $\phi(\sigma)$ such that
$\phi(\sigma)[[a_{\xi}]]=[a_{\sigma^{-1}(\xi)}]$. 
We say that such a
homeomorphism is induced by $\sigma$.
If $f\in C_I(K)$ depends on $S_f\subseteq \omega_2$,
then $f\circ \phi(\sigma)$ depends on
$\sigma[S_f]$. If two bijections $\sigma, \sigma'$ 
agree on $S_f$, then $f\circ \phi(\sigma)=f\circ \phi(\sigma')$.
If $\sigma, \sigma':\omega_2\rightarrow\omega_2$ are two
bijections, then $\phi(\sigma)\circ\phi(\sigma')=\phi(\sigma'\circ\sigma)$.
\end{proposition}
\begin{proof}
Clearly the bijection uniquely determines an isomorphism
of the algebra $Fr(\omega_2)$ which sends $a_\xi$ to
$a_{\sigma(\xi)}$. Now note that by the
Sikorski extension theorem there is an extension to
a homomorphism from $Co(Fr(\omega_2))$ into
$Co(Fr(\omega_2))$. By the density of $Fr(\omega_2)$ in $Co(Fr(\omega_2))$
we get that the extension $h(\sigma)$ must be an isomorphism onto
$Co(Fr(\omega_2))$ satisfying $$h(\sigma)(\{\sup\{b_i:i\in \N\})=
\sup\{h(\sigma)(b_i):i\in \N\}$$
for any $b_i\in Fr(\omega_2)$ and $i\in \N$.
It follows that $h(\sigma)[Co(A)]=Co[\sigma[A]]$.

Let $\phi(\sigma)$ denote the dual mapping to $h(\sigma)$
obtained via the Stone duality. It is clear that it is a homeomorphism
of the Stone space $K$ of $co(Fr(\omega_2))$.  By its definition
$\phi(\sigma)(u)=h(\sigma)^{-1}[u]$, so
$\phi(\sigma)[[a_\xi]]=\{h(\sigma)^{-1}[u]: u\in [a_\xi]\}=
\{h(\sigma)^{-1}[u]: a_\xi\in u\}=\{v: h(\sigma)^{-1}(a_\xi)\in v\}=\{v: a_{\sigma^{-1}(\xi)}\in v\}=[a_{\sigma^{-1}(\xi)}]$
as required.

Now if $x\cap co(\sigma[S_f])=y\cap co(\sigma[S_f])$, then 
$x\in [a]$ if and only if $y\in [a]$ for any $a\in co(\sigma[S_f])$,
it follows from the above that $\phi(\sigma) (x)\in [b]$  if and only if 
$\phi(\sigma)(y) \in [b]$ for any $b\in co(S_f)$,
hence $f\circ \phi(\sigma)(x)=f\circ \phi(\sigma)(y)$ for such $x,y$.
In other words $f\circ \phi(\sigma)$ depends on $\sigma[S_f]$.

To prove  next part of the proposition take any $x\in K$. $h(\sigma)$
agrees with $h(\sigma')$ on $Co(S_f)$ by the uniqueness of the $h(\sigma)$ and
$h(\sigma')$. 
Note that if two functions agree on a set $A$, then the preimages of  sets
with respect to them have the same intersections with $A$.
So, if $b\in co(S_f)$, then  $\phi(\sigma) (x)\in [b]$  if and only if $b\in h(\sigma)^{-1}[x]$
 if and only if $b\in h(\sigma')^{-1}[x]$ if and only if
$\phi(\sigma')(x) \in [b]$, in other words $\phi(\sigma)(x)\cap Co(S_f)=
\phi(\sigma')(x)\cap Co(S_f)$,
hence $f\circ \phi(\sigma)(x)=f\circ \phi(\sigma')(x)$.

To prove the last part of the proposition, take two
bijections $\sigma,\sigma':\omega_2\rightarrow\omega_2$.
Given an $x\in S(Co(\omega_2))$, by
 the definition
$(\phi(\sigma)\circ\phi(\sigma'))(x)=\sigma^{-1}[(\sigma')^{-1}[x]]=
\sigma^{-1}[\{a\in Co(\omega_2): \sigma'(a)\in x\}]=
\{b\in Co(\omega_2): (\sigma'\circ\sigma)(b)\in x\}=
(\sigma'\circ\sigma)^{-1}[x]=\phi(\sigma'\circ\sigma)(x)$.

\end{proof}

\begin{lemma}\label{finitehomeomorphism} Let $\sigma:\omega_2\rightarrow\omega_2$
be a bijection and $\{f_1, ..., f_k\}\subseteq C_I(K)$.
Then $\nabla\{f_1,...f_k\}=\nabla\{f_1\circ \phi(\sigma), ..., f_k\circ \phi(\sigma)\}$.
\end{lemma}
\begin{proof}
Any point of $\nabla\{f_1, ..., f_k\}$ is of the form $\Pi\{f_1, ..., f_k\}(x)$ for
some $x\in K$, but this point is also of the form 
$\Pi\{f_1\circ\phi(\sigma), ..., f_k\circ \phi(\sigma)\}(\phi(\sigma)^{-1} (x))$ which is in 
$\nabla\{f_1\circ \phi(\sigma), ..., f_k\circ \phi(\sigma)\}$. Conversely
any point of $\nabla\{f_1\circ \phi(\sigma), ..., f_k\circ \phi(\sigma)\}$ is of the form 
$\Pi\{f_1\circ \phi(\sigma),...f_k\circ \phi(\sigma)\}( x)$
for some $x\in K$, but this point  is also of the form 
$\Pi\{f_1, ..., f_k\}(\phi(\sigma)(x))$ which is in 
$\nabla\{f_1, ..., f_k\}$.

\end{proof}

\section{An auxiliary partial order of approximations}

\begin{definition}\label{definitionSigma} Suppose that
$\xi<\omega_2$ and $A\subseteq B\subseteq \omega_2$. $\Sigma_{\xi,A}(B)$
is the set of all  bijections of $\omega_2$ satisfying
$$\sigma[A]\cap (A\setminus\xi)=\emptyset,\ \  \sigma|(\xi\cup(\omega_2\setminus B))=Id_{\xi\cup(\omega_2\setminus B)}.$$
\end{definition}

An ideal of subsets of a set $A$
is a family of subsets of $A$ which contains all finite
subsets of $A$ and  is closed under taking finite unions and 
taking subsets of its elements. If $\mathcal J$ is a family of  sets then
$\langle \mathcal J \rangle$ is the ideal 
of subsets of $\bigcup \mathcal J$ generated by $\mathcal J$, i.e., the family of
all subsets of finite unions of elements of $\mathcal J$.
We say that an ideal $\mathcal I$ of subsets of $A$ is proper if it does not contain $A$;
it is countably generated if  there is a countable $\mathcal J\subseteq \mathcal I$ such
that $\langle \mathcal J\rangle =\mathcal I$.

\begin{definition}\label{definitionP}
We define a partial order $\PP$ which consists
of conditions of the form $p=(A_p,\mathcal F_p, \mathcal I_p)$
such that
\begin{enumerate}
\item  $A_p\in [\omega_2]^{\leq \omega}$, $[\alpha,\alpha+\omega)\subseteq A_p$ for all $\alpha\in A_p$
\item $\{d_\xi: \xi\in A_p\}\subseteq \mathcal F_p\in [C_I(K)]^{\leq \omega}$,
\item $\mathcal I_p$ is a proper, countably generated ideal of subsets of $A_p$
\item Each $f\in \mathcal F_p$ depends on a countable
subset $S_f\in \mathcal I_p$.
\item For every $f_1,..., f_k\in \mathcal F_p$, for every $\xi\in A_p$
and $A\in \mathcal I_p$ 
 there is $\sigma\in \Sigma_{\xi, A}(A_p)$,  with
$$f_1\circ \phi(\sigma),...,f_k\circ\phi(\sigma)\in \mathcal F_p.$$

\item $\nabla \mathcal F$ is a connected compact space.
\end{enumerate}
We let $p\leq q$ if and only if $A_p\supseteq A_q$,  $\mathcal F_p\supseteq \mathcal F_q$ and 
$\mathcal I_p\supseteq \mathcal I_q$ for $p,q\in \PP$.
\end{definition}

\begin{definition}\label{isomorphicP}
Let $p,q\in \PP$ we say that they are isomorphic if and only if
$A_p\cap A_q<A_p\setminus A_q<A_q\setminus A_p$ or
$A_p\cap A_q<A_q\setminus A_p<A_p\setminus A_q$ and there is
a bijection $\tau:\omega_2\rightarrow \omega_2$ such that 
$\tau[A_p]=A_q$ and $\tau|(A_p\cap A_q)=Id_{A_p\cap A_q}$
and 
\begin{enumerate}
\item $\mathcal F_q=\{ f\circ \phi(\tau): f\in \mathcal F_p\}$,
\item $\mathcal I_q=\{\tau[A]: A\in \mathcal I_p\}$.
\end{enumerate}
\end{definition}

\begin{lemma}\label{amalgamationP}
Suppose that $p,q\in \PP$ are isomorphic.
Then $r=(A_p\cup A_q, \mathcal F_p\cup \mathcal F_q, \langle \mathcal I_p\cup \mathcal I_q\rangle)$ is
a condition of $\PP$ and $r\leq p,q$.
\end{lemma}

\begin{proof}
We only need to prove that $r\in \PP$. 
The conditions (1), (2) (4) of \ref{definitionP} are clear.
To prove (3) we need to note
the properness of the ideal $\langle \mathcal I_p\cup
\mathcal I_q\rangle$. Note that given $A\in \mathcal I_p$ and $B\in \mathcal I_q$,
if $A_p\cup A_q\subseteq A\cup B$, then $A_q\subseteq \tau[A]\cup B$ which
contradicts the properness of $\mathcal I_q$ using (2) of \ref{definitionP}.
Now we will  prove (5).
Assume that 
$$A_p\cap A_q<A_p\setminus A_q<A_q\setminus A_p\not=\emptyset$$
and that $\xi_0=\min(A_p\setminus A_q)$ and $\eta_0=\min(A_q\setminus A_p)$.
Let $\tau:\omega_2\rightarrow \omega_2$ be a bijection of $\omega_2$ such that
$\tau^2=\tau$, $\tau[A_p]=A_q$ and $\tau\restriction A_p\cap A_q=Id_{A_p\cap A_q}$
which witnesses that $p$ and $q$ are isomorphic, that is, in
particular that  $F_q=\{f\circ \phi(\tau):
f\in \mathcal F_p\}$ and  $\mathcal I_q=\{\tau[A]: A\in \mathcal I_p\}$.
These properties of $\tau$ can be easily arranged as $f\circ \phi(\tau)=f\circ \phi(\tau')$
whenever $\tau$ and $\tau'$ agree on $A_r$ and $f\in \mathcal F_r$ by \ref{generalbijections}.

Consider a finite subset of $\mathcal F_r$, a $\xi\in A_r$ and an element
of $\mathcal I_r$ from (5).  We may assume, by adding new elements
if necesssary  that this finite set consists of elements
$f_1, ..., f_m\in \mathcal F_p$ and $g_1, ..., g_k\in \mathcal F_q$ for $m,k\in\N$
and that the element of $\mathcal I_r$ is of the form $A\cup B$ where 
$A\in \mathcal I_p$ and $B\in \mathcal I_q$.
We may  assume  that $S_{f_1}\cup...\cup S_{f_m}\subseteq A$
and $S_{g_1}\cup...\cup S_{g_k}\subseteq B$ and that $\tau[A]=B$, in particular that
$A\cap A_p\cap A_q=B\cap A_p\cap A_q$.
We need to find $\sigma\in \Sigma_{\xi,A\cup B}(A_r)$ such that
$f_1\circ \phi(\sigma), ..., f_m\circ \phi(\sigma), g_1\circ \phi(\sigma), ..., g_k\circ \phi(\sigma)
\in\mathcal F_r$.

\vskip 6pt
\noindent{\bf Case 1.} $\xi\in A_p\cap A_q$.\par

\noindent Let 
$\sigma_1\in \Sigma_{\xi_0,  A}(A_p)$
be such that 
 $f_1\circ \phi(\sigma_1), ..., f_k\circ \phi(\sigma_1)
\in\mathcal F_p$.
 Let $\sigma_2\in \Sigma_{\xi,  B}(A_q)$ be such that
 $g_1\circ \phi(\sigma_2), ..., g_k\circ \phi(\sigma_2),
 \mathcal F_r$.

Define a function $\sigma': A\cup B\rightarrow A_p\cup A_q$  so that
$\sigma'(\alpha)=\sigma_1(\alpha)$ for $\alpha\in A\setminus B$
and $\sigma'(\alpha)=\sigma_2(\alpha)$ for $\alpha\in B$.

We have that $\sigma'|(B\cap \xi)=\sigma_2|(B\cap\xi)=Id_{B\cap\xi}$.
Now calculate:
$$\sigma'[(A\cup B)\setminus\xi]\cap (A\cup B)=[\sigma_1[A\setminus\xi_0]\cap (A\setminus\xi_0)]\cup
[\sigma_2[B\setminus\xi]\cap B]=\emptyset.$$
Note also that $\sigma'$ is an injection since it is a union of
two injections whose ranges are disjoint. Finally we will extend $\sigma'$ to a bijection
$\sigma$
of $\omega_2$ which is in $\Sigma_{\xi,A\cup B}(A_r)$. For this we need the following:

{\bf Claim:} $(A_q\cup A_p)\setminus (\xi\cup A\cup B)$ and
$(A_q\cup A_p)\setminus (\xi\cup \sigma[A\cup B])$ are  infinite.

Proof of the claim:
Case A. $(A\cup B)\setminus \xi_0$ is finite.

Then, as we assume that $A_p\setminus A_q$ is nonempty and so infinite by 1) of \ref{definitionP}, 
we have that
$[\xi_0,\xi_0+\omega)\subseteq A_p\setminus\xi_0$, so both of the
differences must have infinite intersections with the above interval

Case B. $(A\cup B)\setminus \xi_0$ is infinite.

As we have $\tau[A]=B$, this means that both $A\setminus \xi_0$ and $B\setminus\xi_0$
are infinite, in particular $A\setminus \xi_0, \sigma'[A\setminus\xi_0]\subseteq A_p\setminus A_q$ 
are infinite and disjoint. Also we have 
$$A\setminus \xi_0\subseteq (A_q\cup A_p)\setminus (\xi\cup \sigma'[A\cup B]),$$
$$\sigma'[A\setminus \xi_0]\subseteq (A_q\cup A_p)\setminus (\xi\cup A\cup B).$$
So both of the sets from the claim are infinite which completes
the proof of the claim.

Now we use the claim to define a bijection $\sigma''$ of $(A_p\cup A_q)\setminus \xi$
which extends $\sigma'|((A\cup B)\setminus\xi)$. Finally extend it to a bijection $\sigma$ of
$\omega_2$ by adding $Id_{\omega_2\setminus [(A_p\cup A_q)\setminus \xi]}$.
 Note that 
$\sigma'|((A\cup B)\cap\xi)=Id_{(A\cup B)\cap\xi}=\sigma_1|(B\cap \xi)$
and so $\sigma$ can be the identity while restricted to $\xi\cup (\omega_2\setminus A_r)$
resulting in $\sigma\in \Sigma_{\xi, A\cup B}(A_r)$.

Now note that for $i=1, ...,m$ we have
 $f_i\circ \phi(\sigma)=f_i\circ \phi(\sigma_1)\in\mathcal F_p\subseteq \mathcal  F_r$ by \ref{generalbijections} 
because $S_{f_i}\subseteq A$ and $\sigma_1$ agrees with $\sigma$ on $A$.
Similarly 
 for $i=1, ...,k$
 $g_i\circ \phi(\sigma), ..., g_i\circ \phi(\sigma_1),
 \mathcal F_q\subseteq \mathcal  F_r$ by \ref{generalbijections} 
because $S_{g_i}\subseteq A$ and $\sigma_2$ agrees with $\sigma$ on $B$.
Hence $\sigma\in \Sigma_{\xi, A\cup B}(A_r)$ and
$f_1\circ \phi(\sigma), ..., f_k\circ\phi(\sigma),
g_1\circ \phi(\sigma), ..., g_k\circ\phi(\sigma) \in \mathcal F_r$ as required in 5).

\vskip 6pt
\noindent{\bf Case 2.} $\xi\in A_p\setminus A_q$

\noindent Let $\sigma_1\in \Sigma_{\xi,A}(A_p)$ be such that
$f_1\circ \phi(\sigma_1), ..., f_m\circ\phi(\sigma_1) \in \mathcal F_p$.
Let $\sigma_2\in \Sigma_{\xi_0,B}(A_q)$ be such that
$g_1\circ \phi(\sigma_2), ..., g_k\circ\phi(\sigma_2) \in \mathcal F_q$.
As $\sigma_1|\xi_0=Id_{\xi_0}=\sigma_2|\xi_0$ 
there is a bijection $\sigma$ of $\omega_2$ such that $\sigma|A_r=\sigma_1|A_p\cup \sigma_2|A_q$
and $\sigma$ is the identity on the remaining part of $\omega_2$.
Note that $\sigma|A_r\cap \xi=id_{\xi}$ and that 
$$\sigma[A\cup B]\cap (A\cup B)\setminus\xi\subseteq
(\sigma_1[A]\cap A\setminus\xi)\cup(\sigma_2[B] \cap B\setminus\xi_0)=
\emptyset.$$
Hence $\sigma\in \Sigma_{\xi, A\cup B}(A_r)$. Since
$f_1 ,..., f_k$ depend on $A_p$  we have $f_i\circ\phi(\sigma)=f_i\circ\phi(\sigma_1)$
also since
$g_1 ,..., g_k$ depend on $A_q$  we have $f_i\circ\phi(\sigma)=f_i\circ\phi(\sigma_2)$,
hence 
$f_1\circ \phi(\sigma), ..., f_k\circ\phi(\sigma) 
g_1\circ \phi(\sigma), ..., g_k\circ\phi(\sigma) \in \mathcal F_r$ which completes
the proof of (4) in this case.
\vskip 6pt
\noindent{\bf Case 3.} $\xi\in A_q\setminus A_p$\par

\noindent Let $\sigma\in \Sigma_{\xi,B}$ be such that
$g_1\circ \phi(\sigma), ..., g_k\circ\phi(\sigma) \in \mathcal F_q$.
As $\sigma|\xi=Id_{\xi}$ and all elements of $\mathcal F_p$ depend on
$A_p\subseteq \xi$ we have that $f_i\circ \phi(\sigma)=f_i$
hence $f_1\circ \phi, ...f_m\circ\phi, g_1\circ \phi, ...g_k\circ\phi \in \mathcal F_r$.
Also $(A\cup B)\setminus \xi=A\setminus\xi$ as $B\subseteq A_p\subseteq\xi$, so
$\sigma\in \Sigma_{\xi, A\cup B}$ as required.

\vskip 6pt
To prove (6) we will show that for every finite
$f_1, ..., f_m\in \mathcal F_r$ there are $f_1', ..., f_m'\in \mathcal F_q$
such that $\nabla \{f_1,..., f_m\}$ is homeomorphic to $\nabla \{f_1',...,f_m'\}$.
This will be enough by \ref{subproductsconnected} since $\nabla \{f_1',...,f_m'\}$ is a continuous image of
$\nabla \mathcal F_r$ (by projecting), so is connected by (5) of \ref{definitionP}
for $q\in \PP$.

Let $A=S_{f_1}\cup...\cup S_{f_m}$ and $B=S_{f_1'}\cup...\cup S_{f_m'}$.
Let $\sigma_1\in \Sigma_{\eta_0, \tau[A]}(A_q)$
Define $\sigma': A\cup B\rightarrow A_p\cup A_q$ by
$\sigma'(\alpha)=\sigma_1(\alpha)$ if $\alpha\in B$ and
$\sigma'(\alpha)=\tau(\alpha)$ if $\alpha\in A$. $\sigma$ is well-defined
because both $\sigma_1$ and $\tau$ are the identity on $A_p\cap A_q$.
As $\sigma_1|B\setminus A$ and $\tau|A\setminus B$ have disjoint ranges
it follows that $\sigma'$ is an injection. 
So $\sigma'$ can be extended to a bijection $\sigma$
of $\omega_2$.

Now note that for $i=1, ...,k$ we have $f_i\circ \phi(\tau)=f_i\circ \phi(\sigma)$ by \ref{generalbijections} as
$\tau$ and $\sigma$ agree on $S_{f_i}\subseteq A$, so
using (2) of \ref{isomorphicP} we have
$f_1\circ \phi(\sigma), ..., f_k\circ\phi(\sigma),
g_1\circ \phi(\sigma), ..., g_k\circ\phi(\sigma)\in \mathcal F_q$.

\end{proof}

\begin{lemma}\label{sigmaclosedP} 
Suppose that $(p_n)_{n\in\N}$ is a sequence of conditions of $\PP$
satisfying $p_{n+1}\leq p_n$ and $p_n=(A_n, \mathcal F_n, \mathcal I_n)$ for each $n\in \N$. Then 
 $$p=(\bigcup_{n\in \N}A_n, \bigcup_{n\in \N}\mathcal F_n, \bigcup_{n\in \N}\mathcal I_n).$$
is a condition of $\PP$ satisfying $p\leq p_n$ for each $n\in\N$.

\end{lemma}
\begin{proof} 
First note that $p\in\PP$.
(1), (2), (3) and (4) are clear. To get (5) 
Note that if  $f_1,..., f_k\in \mathcal F_p$, $A\in \mathcal I_p$ and $\xi\in A_p$, 
then $f_1,..., f_k\in \mathcal F_n$, $A\in \mathcal I_n$
and $\xi\in A_n$ for some $n\in\N$. Then by (5) for $p_n$ 
 there is $\sigma\in \Sigma_{\xi, A}(A_n)$,  with
$f_1\circ \phi(\sigma),...,f_k\circ\phi(\sigma)\in \mathcal F_n.$
So we may use the fact that $\Sigma_{\xi,B}(C)\subseteq \Sigma_{\xi, B}(C')$
whenever $C\subseteq C'$.To get (6) apply  \ref{subproductsconnected}.
The fact that $p\leq p_n$ for each $n\in\N$ is clear.
\end{proof}

\section{The partial order of approximations}

\begin{definition}\label{definitionQ}
We define a partial order $\Q$ which consists
of conditions of the form $p=(A_p,\mathcal F_p, \mathcal I_p, 
\alpha_p, \mathcal X_p, \mathcal P_p)$
such that
\begin{enumerate}
\item $(A_p,\mathcal F_p, \mathcal I_p)\in \PP$
\item  $\alpha_p$ is a countable ordinal,
\item $\mathcal X_p=\{x^p_\beta: \beta<\alpha_p\}$
is a countable dense subset of $\nabla \mathcal F_p$
\item $P_p$ is a countable family of pairs $(L, R)$ of disjoint subsets
$\alpha_p$ such that $\{x_\xi^p: n\in L\}$ and
$\{x_\xi^p: n\in R\}$ are relatively discrete in $\nabla \mathcal F_p$
and 
$${\overline{\{x_\xi^p: \xi\in L\}}}\cap {\overline{\{x_\xi^p: \xi\in R\}}}\not=\emptyset.$$
\end{enumerate}
We let $p\leq q$ if and only if $A_p\supseteq A_q$,  $\mathcal F_p\supseteq \mathcal F_q$,
$\mathcal I_p\supseteq \mathcal I_q$,
$\alpha_p\geq \alpha_q$, $\mathcal P_p\supseteq \mathcal P_q$
and  $x_\beta^p|\nabla \mathcal F_q=x^q_\beta$ for
any $\beta<\alpha_q$.
\end{definition}

The point here is that we will make our spaces $\nabla \mathcal F_p$ quite complicated and
rich in suprema of bounded sequences  from $C(\nabla \mathcal F_p)$, but we will
need to keep promises (that is, preserve elements of $\mathcal P_p$) about not
separating some pairs of countable sets of points. If it is done in a
sufficiently random manner,  all the operators which are not weak multipliers are eliminated
like in \cite{few} or \cite{big}.

\begin{definition}\label{isomorphicQ} Let $p,q\in \Q$. We say that they are isomorphic
if there is a bijection $\tau:\omega_2\rightarrow \omega_2$ which witnesses  that
$(A_p,\mathcal F_p, \mathcal I_p)$ and $(A_q,\mathcal F_q, \mathcal I_q)$ are isomorphic
as elements of $\PP$ and
\begin{enumerate}
\item $\alpha_p=\alpha_q$,
\item $x_\xi^p(f)=x_\xi^q(f\circ \phi(\tau))$ for all $f\in \mathcal F_p$ and all $\xi<\alpha_p=\alpha_q$,
\item If $f\in \mathcal F_p\cap \mathcal F_q$, then $x_\xi^p(f)=x_\xi^q(f)$
for all $\xi<\alpha_p$
\item $\mathcal P_q=\mathcal P_p$.

\end{enumerate}
\end{definition}

\begin{lemma}\label{amalgamationQ}
Suppose that $p,q\in \Q$ are isomorphic. Then there is 
$r=(A_p\cup A_q,\mathcal F_p\cup \mathcal F_q, \langle \mathcal I_p\cup \mathcal I_q\rangle, 
\alpha_p+\omega, \mathcal X, \mathcal P_p )$ which is a condition of $\Q$ stronger
than both $p$ and $q$ where $\mathcal X=\{x^r_\beta:\beta<\alpha_r\}$ consists of points
which for every $f\in \mathcal F_q$ satisfy:
$$x_\xi^r(f)=x_\xi^r(f\circ \phi(\tau))\ \ \hbox{for all}\  f\in \mathcal F_p,\ \xi<\alpha_p=\alpha_q.$$
\end{lemma}
\begin{proof} By \ref{amalgamationP} 
$(A_p\cup A_q,\mathcal F_p\cup \mathcal F_q, \langle \mathcal I_p\cup \mathcal I_q\rangle)\leq 
(A_p,\mathcal F_p, \mathcal I_p), (A_q,\mathcal F_q, \mathcal I_q)$.
Given $\beta<\alpha_p$ find $u\in K$ such that $(\Pi\mathcal F_p)(u)=x_\beta^p$.
Let $v=\phi(\tau)^{-1}(u)$, then 
for all $f\in\mathcal F_p$ we have $f\circ \phi(\tau)(v)=f(u)=x_\beta^p(f)=x_\beta^q(f\circ\phi(\tau))$, and so
$(\Pi\mathcal F_q)(v)=x_\beta^q$ by 2) of \ref{isomorphicQ} and 2) of
\ref{isomorphicP}. Note that $v\cap Co(A_p\cap A_q)=u\cap Co(A_p\cap A_q)$
because $\tau$ is the identity on $A_p\cap A_q$.
So by \ref{ultrafilters} there is an ultrafilter $w\in K$ which includes
$u\cap Co(A_p)$ and $v\cap Co(A_q)$, hence $x^\beta_r=(\Pi\mathcal F_r)(w)$ belongs to
$\nabla \mathcal F_r$ and satisfies the requirement of the lemma.
So we are allowed for $\beta<\alpha$ to define $x_\beta^r(f)=x_\beta^p(f)$ if $f\in \mathcal F_p$ and
$x_\beta^r(f)=x_\beta^q(f)$ if $f\in \mathcal F_q$. 

Define $x^r_{\alpha+n}$s to be points of some dense subset of
$\nabla \mathcal F_r$ so that (3) of \ref{definitionQ} will
be satisfied. So to finish the proof we need to show  (4) of
\ref{definitionQ}, i.e., that the promises are preserved, i.e., that
for each $(L,R)\in \mathcal P_r$
  $\{x_\xi^r: n\in L\}$ and
$\{x_\xi^r: n\in R\}$ are relatively discrete in $\nabla \mathcal F_r$
and 
$${\overline{\{x_\xi^r: \xi\in L\}}}\cap {\overline{\{x_\xi^r: \xi\in R\}}}\not=\emptyset.$$
in $\nabla \mathcal F_r$.
Of course the canonical projections from $\nabla \mathcal F_r$ onto $\nabla \mathcal F_p$
or $\nabla \mathcal F_q$ are continuous and send $x_\beta^r$ to $x_\beta^p$
or $x_\beta^q$ respectively. As the images are discrete, the preimages must be as well
so for each $(L,R)\in \mathcal P_r$
  $\{x_\xi^r: n\in L\}$ and
$\{x_\xi^r: n\in R\}$ are relatively discrete in $\nabla \mathcal F_r$. 
Now assume that $(L,R)\in \mathcal P_r=\mathcal P_p=\mathcal P_q$.
It is enough to prove that for each finite $\mathcal F\subseteq \mathcal F_r$, for every $\varepsilon>0$
there are $\xi\in L$ and $\xi'\in R$ such that $|x_\xi(f)-x_{\xi'}(f)|<\varepsilon$ for all
$f\in\mathcal F$.
Fix $\mathcal F$ and $\varepsilon$ as above.
Let $\tau: \omega_2\rightarrow \omega_2$ be a  bijection like in \ref{isomorphicQ} witnessing 
the isomorphism of $p$ and $q$. Let 
$$\mathcal G=(\mathcal F\cap \mathcal F_p)\cup\{f\circ\phi(\tau^{-1}): 
f\in \mathcal F\cap \mathcal F_q\}\subseteq \mathcal F_p.$$
By (4) of \ref{definitionQ} find  $\xi\in L$ and $\xi'\in R$ such that $|x_\xi^r(f)-x_{\xi'}^r(f)|<\varepsilon$
 holds for each $f\in \mathcal G$. In particular for each $f\in \mathcal F\cap \mathcal F_p$ we have 
$$|x_\xi^r(f)-x_{\xi'}^r(f)|=|x_\xi^p(f)-x_{\xi'}^p(f)|<\varepsilon.$$
But by (2) of \ref{isomorphicQ}, for $f\in \mathcal F\cap \mathcal F_q$ we have that 
$$|x_\xi^r(f)-x_{\xi'}^r(f)|=|x_\xi^q(f\circ\phi(\tau^{-1})\circ\phi(\tau))-
x_{\xi'}^q (f\circ\phi(\tau^{-1})\circ\phi(\tau))|=$$ $$=|x_\xi^p(f\circ\phi(\tau^{-1}))-x_{\xi'}^p
 (f\circ\phi(\tau^{-1}))|=|x_\xi^p(g)-x_{\xi'}^p (g)|<\varepsilon,$$
where $g$ is some element of $\mathcal G$, which proves that (4) of \ref{definitionQ} holds for $r$
and completes the proof of the lemma.
\end{proof}

\begin{lemma}\label{cardinalityP}
Assume CH. Let $A\subseteq\omega_2$ be a countable set. There are $\omega_1$
conditions $p$ of $\PP$ with $A_p=A$.
\end{lemma}
\begin{proof}
By \ref{cardinality} there are $\omega_1$ functions in $C_I(K)$ 
which depend on $A$, so using the fact that 
there are continuum countable subsets of $\omega_1$ we see that
there are $\omega_1$ possibilities for the family $\mathcal F_p$.
There are also continuum many possibilities for the countable
set in $\wp(A)$ which  generates $\mathcal I_p$. Similar argument
shows that there are continuum many possibilities for $P_p$ and of course
for $\alpha_p$. 
Note that given a countable $\mathcal F_p\subseteq C_I(K)$, 
$\nabla \mathcal F_p$ is completely determined as a subset of
$[0,1]^{\mathcal F_p}$ which has cardinality continuum, so we have
again continuum many possibilities for its countable dense subsets $\mathcal X_p$.
This completes the proof.
\end{proof}

\begin{lemma}\label{omega2ccQ} Assume CH.
$\Q$ satisfies the $\omega_2$-c.c.
\end{lemma}
\begin{proof}
Let $(p_\xi:\xi<\omega_2)$ be a family of elements of $\Q$. 
Using the $\Delta$-system lemma for the family $(A_{p_\xi}: \xi<\omega_2)$,
which holds under CH, we may assume that these sets form a $\Delta$-system
with root $\Delta<A_{p_\xi}\setminus \Delta$ of the same order type
$\theta<\omega_1$ for each $\xi<\omega_2$. Again applying CH we may
assume that $A_{p_\xi}\setminus \Delta<A_{p_\eta}$ for every $\xi<\eta<\omega_2$

Consider the unique order-preserving bijections $\tau_\xi: A_{p_\xi}\rightarrow \theta$.
And conditions of $\Q$ of the form
$q_\xi=(\theta, \mathcal F_\xi, \mathcal I_\xi, \alpha_\xi, \mathcal X_\xi, \mathcal P_\xi)$ where
\begin{itemize}
\item $\mathcal F_\xi=\{ f\circ \phi(\tau_\xi): f\in \mathcal F_{p_\xi}\}$,
\item $\mathcal I_\xi=\{\tau_\xi[A]: A\in \mathcal I_{p_\xi}\}$.
\item $\alpha_\xi=\alpha_{p_\xi}$,
\item $\mathcal X_\xi=\{x^\xi_\beta: \beta<\alpha_\xi\}$ is a subset of $\nabla \mathcal F_\xi$ satisfying
$x_\beta^\xi(f\circ\phi(\tau_\xi))=x_\beta^{p_\xi}(f)$ 
for all $f\in \mathcal F_p$. 
\item $\mathcal P_\xi=\mathcal P_{p_\xi}$.
\end{itemize}
Using  CH by \ref{cardinalityP} there are at most 
$\omega_1$ conditions of $\Q$ with the first coordinate equal to $\eta$. Thus
we have $q_\xi=q_{\xi'}$ for distinct $\xi, \xi'<\omega_2$. Now an extension
$\tau$ of $\tau_{\xi'}^{-1}\circ \tau_{\xi}$ witnesses the isomorphism between $p_\xi$ and $p_{\xi'}$.
\end{proof}

\begin{lemma}\label{sigmaclosedQ}
Suppose that $(p_n)_{n\in\N}$ is a sequence of conditions of $\Q$
satisfying $p_{n+1}\leq p_n$ and $p_n=(A_n, \mathcal F_n, \mathcal I_n, \alpha_n,
\mathcal X_n, \mathcal I_n)$ for each $n\in \N$. Then there is
 $$p=(\bigcup_{n\in \N}A_n, \bigcup_{n\in \N}\mathcal F_n, 
\bigcup_{n\in \N}\mathcal I_n,\alpha, \mathcal X, \bigcup_{n\in \N} P_n)$$
which is a condition of $\PP$ satisfying $p\leq p_n$ for each $n\in\N$
where $\alpha=\sup_{n\in \N}\alpha_n$ and  $\mathcal X=\{x_\beta: \beta<\alpha\}$ is a subset of $\nabla  \bigcup_{n\in \N}\mathcal F_n$ satisfying
$x_\beta(f)=x_\beta^{p_n}(f)$ for $f\in \mathcal F_n$ . In particular
$\Q$ is $\sigma$-closed.
\end{lemma}

\begin{proof}
First we need to prove the existence of
$x_\beta^p$ as in the lemma. Let $u_n\in K$ be such that
$(\Pi \mathcal F_n)(u_n)=x_\beta^{p_n}$. Let $u\in K$ be a complete accumulation point 
of $u_n$'s in $K$. We claim that $x_\beta^{p}=(\Pi \mathcal F_n)(u)$ works. 
For $f\in F_n$, for $k\geq n$ we have $f(u_k)=x^{n}_\beta$ by the assumption that
$p_k\leq p_n$, so $f(u)=x^{n}_\beta$ must hold as well.

Now we claim that $p\in \PP$ and $p\leq p_n$ for each $n\in \N$.
It is enough to prove the latter.

(1), (2), (3) and (4) are clear. To get (5) 
Note that if  $f_1,..., f_k\in \mathcal F_p$, $A\in \mathcal I_p$ and $\xi\in A_p$, 
then $f_1,..., f_k\in \mathcal F_n$, $A\in \mathcal I_n$
and $\xi\in A_n$ for some $n\in\N$. Then by (5) for $p_n$ 
 there is $\sigma\in \Sigma_{\xi, A}(A_n)$,  with
$f_1\circ \phi(\sigma),...,f_k\circ\phi(\sigma)\in \mathcal F_n.$
So we may use the fact that $\Sigma_{\xi,B}(C)\subseteq \Sigma_{\xi, B}(C')$
whenever $C\subseteq C'$.
To get (6) apply  \ref{subproductsconnected}.
\end{proof}

\section{Adding suprema of disjoint sequences of functions}

Suppose that $L$ is a compact space.
We say that functions $f_n:L\rightarrow \R$ for $n\in \N$ are pairwise disjoint if 
$f_n(x)f_{n'}(x)=0$ for all distinct $n,n'\in\N$ and all $x\in L$.
$GR(f)$ will denote the graph of a function $f$.
We need a simple lemma about pointwise sums of pairwise disjoint sequences of functions:

\begin{lemma}\label{lemmadisjointsums}
Suppose $f_n:L\rightarrow [0,1]$ for $n\in\N$ are pairwise disjoint. Then:
\begin{enumerate}
\item If $(x,t)\in GR(\sum_{n\in \N} f_n)$, then there is $n'\in \N$ such that
for all $m>n'$ we have $(x,t)\in  GR(\sum_{n=0}^m f_n)$
\end{enumerate}
\end{lemma}

\begin{lemma}\label{connectedgraph} Suppose $L$ is a metrizable, compact and connected  space and
that for each $k\leq m\in\N$ the sequences of functions
$(f_n^k)_{n\in \N}$ in $C_I(L)$ are pairwise disjoint. 
Let $F: L\rightarrow[0,1]^m$ be defined by
$$F(x)=(\sum_{n\in \N} f_n^1(x),..., \sum_{n\in \N} f_n^m(x))=\prod_{i\leq m}\sum_{n\in \N} f_n^i(x).$$
for every $x\in L$. Then
the closure of the graph of $F$ is a connected subspace of $L\times[0,1]^m$.
\end{lemma}

\begin{proof} 
Let $X$ be the closure of the graph $GR(F)$ of $F$.
For $l_1, ..., l_{m}\in \N$ define 
 $F^{l_1, ..., l_{m}}: L\rightarrow[0,1]^m$ by
$$F^{l_1,  ...,l_{m}}(x)=(\sum_{n=0}^{l_1} f_n^1(x), ..., \sum_{n=0}^{l_{m}} f_n^m(x)),$$
for every $x\in L$. Now consider the set
$$Y=\bigcap_{(l_1, ..., l_{m_0})\in \I^{m}} \overline{\bigcup_{k\in \N}GR(F^{l_1(k), ..., l_{m}(k)})},$$
where $\I$ stands for the set of all strictly increasing sequences of positive integers.

Note that every sum in the above set is a connected set because the functions
are continuous and there is a point where all of them are zero by compactness of $L$ and
the disjointness of the functions. It follows that the closures of the sums are connected
and so that $Y$ is connected as well. Thus it is enough to prove that $X=Y$.

First prove that $X\subseteq Y$. Let  $y=(x,t_1, ..., t_m)\in  X=\overline{GR(F)}$. 
We will show that
$y\in Y$.  Let $(l_1, ..., l_{m})\in \I^{m}$. As $L$ is metrizable, we can find a sequence
$(y_n)_{n\in\N}$ converging to $y$ such that $y_n\in GR(F)$. 
By \ref{lemmadisjointsums} for each
$n\in\N$ there is $n'\in \N$ such that $y_n\in GR(F^{l, ..., l})$ for all $l>n'$ and hence
there is $k\in \N$ such that $y_n\in GR(F^{l_1(k),....l_{m_0}(k)})$.
Hence all $y_n$s are in all the sums appearing in the definition of $Y$ hence
$y$ is in all of their closures and so is in $Y$.

Now we show that $Y\subseteq X$.
Suppose $y=(x,t_1,..., t_m)\not \in X=\overline {GR(F)}$ and
let $U\times V_1\times ...\times V_m$ be an
open neighbourhood of $y$ disjoint from $GR(F)$. By considering a slightly
smaller set we may assume that $\overline U\times \overline V_1\times ...\times 
\overline V_m$ is  disjoint from $GR(F)$. Moreover assume that
$V_i$ is separated from $0$ if $t_i\not=0$ for $1\leq i\leq m$.

Let $k\in\N$. We will find
 $k'(k)>k$ such that
 for some choice of $l_1(k),..., l_{m}(k)\in \{k+1, k'(k)\}$
 the graph 
$GR(F^{l_1(k),... l_{m}(k)})$
is disjoint from $U\times V_1\times, ...,  V_m$.

As $k$ is arbitrary,  we get that
$\bigcup_{k\in \N} GR(F^{l_1(k),... l_{m_0}(k)})$  is disjoint from
 $U\times V_1\times...\times, ..., \times V_m$ proving that $y$ is not in $Y$.

Note that if all $t_1 ,..., t_m$ were bigger than $0$, hence all
$V_1,...V_m$ separated from $0$, and
$GR(F^{k+1,... k+1})$ intersected $U\times V_1\times, ..., V_m$, then the graph $GR(F)$ intersected it as well,
because $0\leq F^{k+1, ...,  k+1}(x)= F(x)$ if all the coordinates
of $F^{k+1,... k+1}(x)$ are bigger than $0$ by the disjointness of the functions
in the original sequences.

On the other hand if all $t_1, ..., t_m=0$, then $F(x)=(s_1,..., s_m)\not=(0,...0)$, so
there is $1\leq i\leq m$ with $s_i\not\in V_i$. Hence there is 
 $k'>k$ such that $\sum_{n=1}^{k'} f_n^{i}(x)=s_i$
and so  $GR(F^{k',... k'})$ is disjoint from
$U\times V_1\times...\times V_{j_0}
\times V_{j_0+1}\times... \times V_m$.

So, some of the values of $t_1 ,..., t_m$ are $0$ and some are not.
We will continue the proof under the
assumption that there is $1\leq m_0\leq m$ be such that $t_1= ... =t_{m_0}=0$ and
$t_{m_0+1}, ..., t_{m}>0$. This of course can easily be transformed into the
general case with a different configuration.

Recall that we have  that $0\not\in\overline V_{m_0+1}, ...,
\overline V_m$. Let $F_{m_0}^{k+1}: L\times [0,1]^{m-m_0}$ be defined by
$$F_{m_0}^{k+1}(x) =(\sum_{n=1}^{k+1} f_n^{m_0+1}(x),..., \sum_{n=1}^{k+1} f_n^m(x))$$
Consider 
$$E=(F_{m_0}^{k+1})^{-1}[\overline V_{m_0+1}\times...\times \overline V_m]\cap U$$
If $E$ is empty then 
$(F^{k+1,..., k+1})^{-1}[\overline V_{1}, ..., \overline V_m]\cap U$ is empty what is required.
If $E$ is nonempty, consider
$$\mathcal U=\{ (f^1_n)^{-1}[[0,1]\setminus\overline V_1], ...,
 (f^{m_0}_n)^{-1}[[0,1]\setminus\overline V_{m_0}]:
n\in \N\}.$$
$\bigcup\mathcal U$ must include $E$ because otherwise there would be
$x\in U$ with $F(x)\in  V_1\times... \times V_{m_0}\times\overline V_{m_0+1}... \times \overline
V_{m}$ which would contradict the choice of $U\times V_1\times... \times V_m$. 
But $E$ is compact, so there is a finite $\mathcal U'\subseteq \mathcal U$ which covers $E$. Let
$k'>k$ be such that $E\subseteq \bigcup \{ (f^1_n)^{-1}[[0,1]\setminus\overline V_1], ..., (f^{m_0}_n)^{-1}[[0,1]\setminus\overline V_{m_0}]: 1\leq n\leq k'\}.$
This means that the graph 
$F^{k',...,k', k+1,...k+1}$ is disjoint from $U\times V_1\times...\times V_{m_0}
\times\overline V_{m_0+1}\times... \times \overline V_m$ as required.

\end{proof}

\begin{definition}\label{definitionD}
Suppose that  $L$ is a compact space and  $f_n:L\rightarrow [0,1]$ are
pairwise disjoint continuous functions for $n\in A\subseteq \N$. Then
$$D((f_n)_{n\in A})=\bigcup\{U: U\ \hbox{\rm is open and}
\  \{n\in A: supp(f_n)\cap U\not=\emptyset\}\ \hbox{\rm is finite}\}$$
Suppose that    $f_n^i:L\rightarrow [0,1]$ for each $n\in A\subseteq \N$ are
pairwise disjoint continuous functions for each $i\in B\subseteq \N$. Define $D((f_n^i)_{n\in A, i\in B})=
\bigcap_{i\in B}D((f_n)_{n\in A})$.
\end{definition}

\begin{lemma}\label{lemmaD} Let $L$ be a compact space. Suppose that
for each $i\in\N$ the sequence $(f_n^i)_{n\in \N}$ consists of pairwise disjoint
 continuous functions from $L$ into $[0,1]$.
\begin{enumerate}
\item  $D((f_n^i)_{i,n\in\N})$ is dense in $L$
\item For each $i\in \N$ the function $\sum_{n\in N}f_n^i$ is continuous on $D((f_n^i)_{n\in\N})$.
\end{enumerate}
\end{lemma}
\begin{proof}
As shown in \cite{few} $D((f_n)_{n\in \N})$ is always a dense and open subset of $K$.
So $D((f_n^i)_{i,n\in\N})$ is dense by the Baire category theorem for
compact Hausdorff spaces. Note that on $D((f_n^i)_{n\in\N})$  the infinite sum $\sum_{n\in N}f^i_n$ is
equal to some finite subsum, and so, is continuous.
\end{proof}

\begin{definition}\label{strongextension}[4.2.\cite{few}]
Suppose that
$L$ is a compact space, $M\subseteq L\times[0,1]^\N$
and for each $i\in B\subseteq\N$ the sequence $(f_n^i)_{n\in A}$ consists of pairwise disjoint
 continuous functions from $L$ into $[0,1]$ for some $A,B\subseteq \N$.
We say that $M$ is an extension of $L$ by 
$(f_n^i)_{n\in \N, i\in B}$ if and only if
 $M$ is the closure of the graph of 
the restriction 
$$\prod\{\sum_{n\in A}f_n^i:i\in B\}|D((f_n^i)_{n\in A, i\in B}).$$
Moreover we  say that $M$ as above is 
a strong extension of $L$ by $(f_n^i)_{n\in A, i\in B}$ if and only
if the graph of $\prod\{\sum_{n\in A}f_n^i:i\in B\}$ is a subset of $M$.
\end{definition}

\begin{lemma}\label{strongconnected}
Suppose that $L$ is a compact and connected space and $M$ is a
strong extension of $L$, then $M$ is compact and connected as well.
\end{lemma}
\begin{proof}
It follows from \ref{connectedgraph} and \ref{strongextension} since
the graph of $\prod\{\sum_{n\in A}f_n^i:i\in B\}$ must be dense in $M$
if it  is a subset of $M$.
\end{proof}

\begin{lemma}\label{strongfinite}
 Suppose that $L$ is compact and metrizable.
Suppose that for each $i\in\N$ the sequence $(f_n^i)_{n\in \N}$ consists of pairwise disjoint
 continuous functions from $L$ into $[0,1]$.
Suppose that for each finite $a\subseteq\N$ the extension of
$L$ by $(f_n^i)_{n\in \N, i\in a}$ is strong. Then, the extension of $L$ by 
$(f_n^i)_{n,i\in \N}$ is strong as well.
\end{lemma}

\begin{proof} Let $D=D((f_n^i)_{n,i\in \N})$.
Suppose that the extension of $L$ by 
$(f_n^i)_{n,i\in \N}$ is not strong and take a point 
$$y=(x,t_1,...)\in GR(\prod\{\sum_{n\in N}f_n^i:i\in \N\})$$
which is not in the closure of 
$$GR(F|D)=GR(\prod\{\sum_{n\in N}f_n^i\restriction D((f_n^i)_{n,i\in \N}):i\in \N\}).$$
Let $U$ be an open neighbourhood of $y$ disjoint from $GR(F|D)$. 
We may assume that $U$ is open basic, so there is
 a finite $a\subseteq \N$ which determines $U$, so $U$ must be disjoint from the projection
of $GR(F|D)$ on theses coordinates, in other words, assuming (which can be done without
loss of generality) that $a=\{1, ..., k\}$ for some $k\in\N$ we get 
$$(x,t_1, ..., t_k)\in GR_k(F)=GR(\prod\{\sum_{n\in N}f_n^j:j\leq k\})$$
and its open neighbourhood $U\subseteq L\times [0,1]^k$  such that
$U$ is disjoint from 
$$GR_k(F|D)=GR(\prod\{\sum_{n\in N}f_n^j\restriction D((f_n^i)_{n,i\in \N}):j\leq k\}).$$

So, to obtain a contradiction with the hypothesis, it is enough to prove that 
$$GR_k(f|D_k)=GR(\prod\{\sum_{n\in N}f_n^i\restriction D((f_n^i)_{n\in \N, i\leq k}):i\leq k\})
\subseteq \overline{GR_k(f|D)},$$
since the left-hand side is dense in $GR_k(F)$ by the assumption of the lemma that finite-dimensional extensions
are strong.
So take any  $(x',t_1', ..., t_k')\in GR_k(f|D_k)$ and 
a neighbourhood $U'\subseteq D_k=D((f_n^i)_{n\in \N, i\leq k})$ of $x'$ and
 neighbourhoods $V_i$ of $t_i'$ for $1\leq i\leq k$. 
All functions $\sum_{n\in \N}f^i_n$ for $1\leq i\leq k$ are continuous in $D_k$
by (2) of \ref{lemmaD}, so we can find $U''\subseteq U'$ such that 
$$x'\in U''\subseteq \bigcap_{1\leq i\leq k} (\sum_{n\in \N}f^i_n)^{-1}[V_i].$$
Now, by (1) of \ref{lemmaD}, take any $y\in D\cap U''$, we have $(y,s_1, ..., s_k)\in GR_k(F|D)$
for some $(s_1, ..., s_k)$ and $(y,s_1, ..., s_k)\in U\times V_1\times ...\times V_k$,
so there is no neighbourhood which can separate $(x',t_1', ..., t_k')$ from
${GR_k(f|D)}$ and so $GR_k(f|D_k)\subseteq  \overline{GR_k(f|D)}$ as required.
\end{proof}

\begin{lemma}\label{strongalmostfinite}
Suppose that $L$ is compact and metrizable.
Suppose that for each $i\leq m $ the sequence $(f_n^i)_{n\in \N}$ consists of pairwise disjoint
 continuous functions from $L$ into $[0,1]$.
Then there is and infinite $A\subseteq \N$ such that for every infinite
$A'$ almost included in $A$ the extension
of $L$ by $(f_n^i)_{n\in A', i\leq m}$ is a strong extension.
\end{lemma}
\begin{proof} Let $d$ be a metric on $L$ compatible with its topology.
We construct $A=\{n_k:k\in\N\}$ so that $n_k<n_{k+1}$
for all $k\in\N$ and  whenever $A'\subseteq \N$ is disjoint with the 
intervals $(n_k, n_{k+1})$ for almost all $k\in\N$, then for every $\varepsilon>0$,
for every $x\in L$ there is 
$$y\in D=\bigcap_{i\leq m}D((f_n^i)_{n\in A', i\leq m})$$
such that $d(x,y)<\varepsilon$ and
$|\sum_{n\in \N}f^i_n(x)-\sum_{n\in \N}f^i_n(y)|<\varepsilon$ for each $i\leq m$.

We construct $n_k$s by induction. Suppose we have constructed $n_1< ...< n_k$.
Using the compactness and the continuity of the functions $\sum_{n\in a}f^i_n$ for
$i\leq m$ and $a\subseteq [0,n_k]$  find a finite open cover $\mathcal U_k$ of $L$ consisting of sets $U$
such that 
$$diam(\sum_{n\in a}f^i_n[U])<1/k, \ \ diam(U)<1/k, \leqno (1)$$
for every $i\leq m$ and every $a\subseteq [0,n_k]$.
Pick $y_U\in U\cap D$ for every $U\in\mathcal U_k$ and define $n_{k+1}$ to be such 
a natural number bigger than $n_k$ that 
$$ \forall U\in \mathcal U_k\  \forall i\leq m\  \forall n\geq n_{k+1}\ \ f_n^i(y_U)=0.\leqno (2)$$
 It can be found because of the disjointness of
the functions. This completes the inductive construction.

Now take any $A'\subseteq \N$ disjoint with the 
intervals $(n_k, n_{k+1})$ for almost all $k\in\N$, fix $\varepsilon>0$
and $x\in L$.  Let $k$ be big enough so that $(n_k, n_{k+1})$ is disjoint from
$A'$, $1/k<\varepsilon$ and if $\sum_{n\in A'}f^i_n(x)>0$, then $f^i_n(x)>0$ for
some $n\leq n_k$ and $n\in A'$.  In particular
$$\forall i\leq m\ \ \sum_{n\in A'}f^i_n(x)=\sum_{n\in a}f^i_n(x),$$
where $a=A'\cap[0, n_k]$. Find $U\in \mathcal U_k$ such that $x\in U$ and
$y_U$ from the construction. In particular $d(y_U,x)<\varepsilon$. For all $i\leq m$ we have 
$$ \sum_{n\in A'}f^i_n(y_U)=\sum_{n\in a}f^i_n(y_U)+
\sum_{n\in A'\cap(n_k, n_{k+1})}f^i_n(y_U)+\sum_{n\in A'\cap [n_{k+1},\infty)}f^i_n(y_U).$$
But the two last terms are zero by the choice of $A'$ and by the choice (2) of
$n_{k+1}$ in the inductive construction, so 
for all $i\leq m$ we have  $\sum_{n\in A'}f^i_n(y_U)=\sum_{n\in a}f^i_n(y_U)$ and hence
$|\sum_{n\in A'}f^i_n(x)-\sum_{n\in A'}f^i_n(y_U)|<1/k<\varepsilon$ for each $i\leq m$
by (1).

\end{proof}

\begin{lemma}\label{strongalmostinfinite}
Suppose that $L$ is compact and metrizable.
Suppose that for each $i\in\N$ the sequence $(f_n^i)_{n\in \N}$ consists of  pairwise disjoint
 continuous functions from $L$ into $[0,1]$.
Then there is and infinite $A\subseteq \N$ such that for every infinite
$A'$ almost included in $A$ the extension
of $L$ by $(f_n^i)_{n\in A', i\in \N}$ is a strong extension.
\end{lemma}

\begin{proof}
Enumerate  all finite subsets of $\N$ as $\{a_k:k\in \N\}$. By induction
construct almost decreasing sequence $(A_k)_{k\in \N}$  of infinite subsets
of $\N$ such that $A_{k+1}$ is almost included in $A_k$ for each $k\in\N$
and $A_k$ satisfies \ref{strongalmostfinite} for $a_k$ instead of $\{1, ...m\}$,
that is, for every $A'$ almost included in $A_k$ the extension of
$L$ by $(\sum_{n\in A'}f^i_n)_{n\in\N, i\in a_k}$ is strong. Now let $A$
be the diagonalization of $A_k$s, that is, an infinite $A\subseteq \N$ which is almost
included in $A_k$ for each $k\in\N$. It is clear that if $A'$ is almost included 
in $A$, then it is almost included in each $A_k$ for $k\in \N$. Now apply lemma
\ref{strongfinite} to conclude that the extension of $L$  by $(f_n^i)_{n\in A', i\in \N}$
is strong.

\end{proof}

\begin{lemma}\label{preservingpromisses}
Suppose that $L$ is compact and metrizable.
Suppose that for each $i\in\N$ the sequence $(f_n^i)_{n\in \N}$ consists of pairwise disjoint
 continuous functions from $L$ into $[0,1]$.
Let $x_n, y_n\in L$ for $n\in\N$ be such that
$$\overline{\{x_n:n\in\N\}}\cap \overline{\{y_n:n\in\N\}}\not=\emptyset.$$
Then there is and infinite $A\subseteq \N$ such that for every infinite
$A'$ almost included in $A$ we have 
$$\overline{\{x_n':n\in\N\}}\cap \overline{\{y_n':n\in\N\}}\not=\emptyset,$$
where $x_n'$ is the point of $GR(\Pi_{i\in \N}\sum_{n\in A'}f^i_n)$ whose $L$-coordinate is
$x_n$ and $y_n'$ is the point of $GR(\Pi_{i\in \N}\sum_{n\in A'}f^i_n)$ whose $L$-coordinate is
$y_n$.
\end{lemma}
\begin{proof} Let $d$ be a metric on $L$ compatible with the topology. 
By going to a subsequence and possibly renumerating the points we may assume that
$$d(x_n, y_n)<1/n.\leqno 1)$$
We construct two strictly increasing sequences $A=\{n_k: k\in \N\}$ 
and $(l_k: k\in \N)$ such that whenever $A'\cap(n_k, n_{k+1})=\emptyset$ then,
for all $i\leq k$ we have
$$|\sum_{n\in A'}f^i_n(x_{l_k})-\sum_{n\in A'}f^i_n(y_{l_k})|<1/k.\leqno 2)$$

This will be enough, since if $A'$ is almost included in $A$, the above will hold for
almost all $k\in \N$; if $\{(x_{l_k}, \Pi_{i\in \N}\sum_{n\in A'}f^i_n(x_{l_k})): k\in\N\}$
were separated from $\{(y_{l_k}, \Pi_{i\in \N}\sum_{n\in A'}f^i_n(y_{l_k})): k\in\N\}$, it could be done by clopen basic sets, which involve only finitely many coordinates, so
there would be $k_0$ such that $\{(x_{l_k}, \Pi_{i\leq k_0}\sum_{n\in A'}f^i_n(x_{l_k})): k\in\N\}$
were separated from $\{(y_{l_k}, \Pi_{i\leq k_0}\sum_{n\in A'}f^i_n(y_{l_k})): k\in\N\}$.
But 2) implies that these points of $L\times[0,1]^{k_0}$ are arbitrarily close.

So let us focus on proving 2).  The construction of $\{n_k: k\in \N\}$ 
and $(l_k: k\in \N)$ is by induction. Suppose that we are done up to $k$.
  $\sum_{n\in a}f^i_n$ are uniformly continuous
for each $i\leq k$ and each $a\subseteq [0,n_k]$. So find  $\delta>0$ such that
if $d(x,y)<\delta$ then $|\sum_{n\in a}f^i_n(x)-\sum_{n\in a}f^i_n(y)|<1/k$
holds for each $i\leq k$ and each $a\subseteq [0,n_k]$. Take $l_{k+1}$ such that
$\delta>1/l_{k+1}$, $l_{k+1}>l_k$ and take $n_{k+1}>n_k$ big enough so that
for all $i\leq k$ we have 
$$\sum_{n\geq n_{k+1}}f^i_n(x_{l_{k+1}})
=\sum_{n\geq n_{k+1}}f^i_n(y_{l_{k+1}})=0.\leqno 3)$$
This completes the inductive construction.

Now, suppose that $A'\subseteq \N$ is such that $A'\cap (n_k, n_{k+1})=\emptyset$.
Put $a=A'\cap [0, n_{k}]$. By 1) and 3) for every $i\leq k+1$ we have
$$|\sum_{n\in A'}f^i_n(x_{l_{k+1}})-
\sum_{n\in A'}f^i_n(y_{l_{k+1}})|=$$
$$|\sum_{n\in a}f^i_n(x_{l_{k+1}})-
\sum_{n\in a}f^i_n(y_{l_{k+1}})+\sum_{n\geq n_{k+1}}f^i_n(x_{l_{k+1}})
-\sum_{n\geq n_{k+1}}f^i_n(y_{l_{k+1}})|=$$
$$|\sum_{n\in a}f^i_n(x_{l_{k+1}})-
\sum_{n\in a}f^i_n(y_{l_{k+1}})|<1/k$$
what was required.

\end{proof}

Given $p\in\PP$
we can add to $\mathcal F_p$ suprema of continuous functions on $\nabla \mathcal F_p$ in such a way
that they cannot be destroyed when we pass to bigger conditions $q\leq p$.

\begin{definition}\label{definitionsupindes}
Suppose that $\mathcal F\subseteq C_I(K)$ and 
$f_n\in C_I(\nabla \mathcal F)$ form a bounded sequence of functions, and let
$h\in C(\nabla \mathcal F)$ be the supremum of  $f_n$s.
Then we say that $h$ is an indestructible supremum 
of $f_n$s if and only if for any $\mathcal G\subseteq C_I(K)$ 
such that $\mathcal F\subseteq \mathcal G$ we have that
$h\circ \pi_{\mathcal F \mathcal G}$ is the supremum of 
$(f_n\circ \pi_{\mathcal F\mathcal G})_{n\in \N}$ in 
 $ C(\nabla \mathcal G)$.
\end{definition}

\begin{lemma}\label{lemmasupindes}
Suppose that $\mathcal F\subseteq C_I(K)$ and 
$f_n\in C_I(\nabla \mathcal F)$ form a bounded sequence of functions, 
Then there is a $g\in C_I(K)$
such that in $C(\nabla (\mathcal F\cup\{g\}))$ there is an indestructible supremum of 
$f_n\circ\pi_{\mathcal F, \mathcal F\cup\{g\}}$s.
\end{lemma}

\begin{proof}
Consider $f_n\circ \Pi \mathcal F:K\rightarrow \R$. Since $K$ is
extremely disconnected, $C(K)$ is a complete lattice
and so  we have $g=sup(f_n\circ \Pi \mathcal F)$ in $C(K)$.
Now, the supremum $h$
of $f_n\circ\pi_{\mathcal F, \mathcal F\cup\{g\}}$s is just taking $g$'s coordinate in $\nabla (\mathcal F\cup\{g\})$.
First we will note that $h\geq f_n\circ\pi_{\mathcal F, \mathcal F\cup\{g\}}$ for
all $n\in\N$. Take $t\in \nabla(\mathcal F\cup\{g\})$ and $x\in K$ such that
$t=\Pi(\mathcal F\cup\{g\})$, we have
$$f_n\circ\pi_{\mathcal F, \mathcal F\cup\{g\}})(t)=
f_n\circ\pi_{\mathcal F, \mathcal F\cup\{g\}}\circ\Pi(\mathcal F\cup\{g\})(x)=
f_n\circ\Pi(\mathcal F)(x)\geq g(x)=t(g)=h(t),$$
as needed.

Now, suppose $\mathcal H=\mathcal F\cup \{g\}\subseteq \mathcal G\subseteq C_I(K)$
and that $h\circ \pi_{\mathcal H \mathcal G}$ is not the supremum of 
$f_n\circ \pi_{\mathcal F \mathcal G}$s in 
$C(\nabla \mathcal G)$. Let  $f\in C_I(\nabla \mathcal G)$ be a function which witnesses it.
That is $f_n\circ \pi_{\mathcal F,\mathcal G}\leq f$ but there is a nonempty open
set $U\subseteq \nabla \mathcal G$ such that $f|U<h\circ \pi_{\mathcal H,\mathcal G}|U$.
Note that then we  have 
$$f_n\circ \Pi \mathcal F=f_n\circ \pi_{\mathcal F ,\mathcal G}\circ \Pi \mathcal G\leq f\circ \Pi \mathcal G.$$
and
$$f\circ \Pi \mathcal G | V< h\circ\pi_{\mathcal H,\mathcal G}\circ \Pi \mathcal G|V=g|V,$$
where $V=(\Pi \mathcal G)^{-1}[U]$. But this contradicts  the fact that $g$ is the supremum of 
$f_n\circ\Pi \mathcal F$s in $C(K)$.
\end{proof}

\begin{lemma}\label{supremasextension}
Suppose $\mathcal F\subseteq C_I(K)$ and that for each $i\in \N$
the sequences $(f_n^i)_{n\in\N}$ of elements
of $C_I(\nabla\mathcal F)$ are pairwise disjoint.
Suppose that all elements of $\mathcal F$
 depend on a set $A\subseteq \omega_2$ and that
$\{d_\xi: \xi\in A\}\subseteq \mathcal F$.
Suppose that $f^i\in C_I(K)$ is the supremum of
$(f_n^i\circ \Pi\mathcal F)_{n\in\N}$. Then $\nabla ( \mathcal F\cup\{f^i:i\in\N\})$
is an extension of $\nabla F$ by $(f_n^i)_{n\in\N}$.
\end{lemma}
\begin{proof}
First we will prove that if $D\subseteq \nabla \mathcal F$ is dense and open in 
$\nabla \mathcal F$, then $(\Pi \mathcal F)^{-1}[D]$ is dense and open in $K$.
Suppose not, then there  are $\xi_1, ..., \xi_n\in\omega_2$ and $\varepsilon_1, ...,
\varepsilon_n\in \{-1,1\}$ such that 
$$[a_{\xi_1}^{\varepsilon_1}]\cap ... \cap [a_{\xi_n}^{\varepsilon_n}]\cap (\Pi \mathcal F)^{-1}[D]
=\emptyset,$$
Where $a^1=a$ and $a^{-1}$ is the complement of $a$ in the Boolean algebra $Fr(\omega_2)$.
As elements of $\mathcal F$ depend on $A$, the characteristic function of
$(\Pi \mathcal F)^{-1}[D]$ depend on $A$, and hence we may assume that 
$\xi_i\in A$ for all $i\leq n$.
$$E=\bigcap _{1\leq i\leq n}\{t\in \nabla\mathcal F: |t(d_{\xi_i})-\varepsilon_i|<1/4\}$$
is a nonempty basic open set in $\nabla \mathcal F$, since $\Pi\{d_{\xi_i}: 1\leq i\leq n\}$s are
onto $[0,1]^{\{d_{\xi_i}: 1\leq i\leq n\}}$. So $E\cap D\not=\emptyset$ by the
density of $D$. But if $x\in (\Pi \mathcal F)^{-1}[E\cap D]$
then $x\in [a_{\xi_i}^{\varepsilon_i}]$ for each $i\leq n$ a contradiction.

In particular we can conclude that for each $i\in \N$ we have that $(\Pi\mathcal F)^{-1}[D((f_n^i)_{n\in \N})]$
is dense open in $K$ and so  
$$(\Pi\mathcal F)^{-1}[\bigcap_{i\in \N}D((f_n^i)_{n\in \N})]=
\bigcap_{i\in\N}(\Pi\mathcal F)^{-1}[D((f_n^i)_{n\in \N})]$$
is dense in $K$ and hence
$\nabla (\mathcal F\cup\{f^i:i\in\N\})$  is the 
closure of 
$$\Pi\mathcal (F\cup\{f^i:i\in\N\})[(\Pi\mathcal F)^{-1}[\bigcap_{i\in\N}D((f_n)_{n\in \N})]].$$
Note that for each $i\in\N$ we have $(\Pi\mathcal F)^{-1}[D((f_n^i)_{n\in \N})]\subseteq$
$D((f_n^i\circ \Pi\mathcal F)_{n\in\N})$, because the preimages of disjoint sets
are disjoint. By \ref{lemmaD} (2), $f^i$ coincides with
$\sum_{n\in \N} f_n^i\circ \Pi\mathcal F$ on $D((f_n^i\circ \Pi\mathcal F)_{n\in\N})$,
so $\nabla (\mathcal F\cup\{f^i:i\in\N\})$  is the 
closure of
$$\Pi\mathcal (\mathcal F\cup\{\sum_{n\in\N}f^i_n\circ \Pi\mathcal F:i\in\N\})[(\Pi\mathcal F)^{-1}[\bigcap_{i\in\N}D((f_n^i)_{n\in \N})]]$$
which is exactly the closure of the graph of $\Pi_{i\in\N}\Sigma_{n\in\N}f^i_n$ restricted
to $\bigcap_{i\in\N}D((f_n^i)_{n\in \N})$ which completes the proof of the lemma.
\end{proof}

\section{The main extension lemma}

\begin{lemma}\label{preimages2-2disjoint} Suppose $f: K\rightarrow L$ is a continuous function
and $(f_n)_{n\in \N}$ is a sequence of pairwise disjoint functions in $C_I(L)$.
Then $(f_n\circ f)_{n\in \N}$ is  pairwise disjoint as well.
\end{lemma}

\begin{lemma}\label{supremacompositions} Suppose that $\phi:K\rightarrow K$ is a homeomorphism
of $K$. Suppose that $f$ is the supremum in $C(K)$ of a 
bounded sequence $(f_n)_{n\in \N}$, then $f\circ \phi$ is the supremum in $C(K)$ of 
$(f_n\circ\phi)_{n\in \N}$.
\end{lemma}
\begin{proof} A homeomorphism of $K$ induces an order-preserving isometry of $C(K)$,
so the supremum must be preserved.
\end{proof}

\begin{lemma}\label{mainextensionlemma}
Let $p$ be in $\Q$. Suppose
that 
\begin{enumerate}[(1)]
\item $(f_n)_{n\in \N}$ is  a sequence of pairwise disjoint functions
in $C_I(\nabla\mathcal F_p)$,

\item $\{\xi_m: m\in \N\}\subseteq \alpha_p$ is such that $f_n(x_{\xi_m}^p)=0$ for each $n,m\in \N$.
\end{enumerate}
Then there is an infinite $A\subseteq \N$ such that for each infinite
$A'\subseteq A$ there is $q_{A'}\leq p$ in $\Q$ such that:
\begin{enumerate}[(a)]
\item there is an $f\in C_I(\nabla \mathcal F_{q_{A'}})$ which is the  indestructible supremum of $(f_i\circ 
\pi_{\mathcal F_p, \mathcal F_{q_A'}})_{i\in A'}$,
\item $({\{x_{\xi_n}^{q_{A'}}: n\in A'\}},{\{x_{\xi_n}^{q_{A'}}: n\not\in A'\}})\in \mathcal P_{q_{A'}}.$

\end{enumerate}
\end{lemma}

\begin{proof} 
First, we will find an auxiliary condition
$s$ using \ref{sigmaclosedQ} for
a sequence $(p_n)_{n\in\N}$ starting with $p_1= p$. 
Let $\Theta=(\Theta_1, \Theta_2):\N\rightarrow \N\times\N$  be a bijection
such that $\Theta_2(n), \Theta_1(n)\leq n$ holds for each $n\in\N$.
The sequence $(p_n)_{n\in\N}$  is constructed by induction together with
a sequence of bijections $\Gamma_n:\N\rightarrow A_{p_n}$.
Having constructed $p_n$ we find 
$$\eta_n=\Gamma_{\Theta_1(n)}(\Theta_2(n))\in A_{p_{\Theta_1(n)}}\subseteq A_{p_n}.$$
Now consider an bijection $\tau_n:\omega_2\rightarrow\omega_2$
such that $\tau_n\restriction \eta_n=Id_{\eta_n}$ and
$\tau_n[A_{p_n}\setminus\eta_n]>  \sup(A_p)$. Construct a condition $p_n'$, transporting all the appropriate
objects by functions induced by $\tau_n$, so that $\tau_n$ witnesses that
$p_n$ and $p_n'$ are isomorphic. Now let $p_{n+1}\leq p_n, p_n'$ be 
the amalgamation like in \ref{amalgamationQ}. In particular for $f\in \mathcal F_{p_n}$
for every $\beta<\alpha_{p_n}$ we have:
$$x^{p_{n+1}}_\beta(f)=x^{p_{n+1}}_\beta(f\circ \phi(\tau_n)).\leqno +)$$
Finally let $\Gamma_{n+1}: \N\rightarrow A_{p_{n+1}}$ be any bijection between
these sets. Let
$s$ be the lower bound of $(p_n)_{n\in\N}$ obtained using \ref{sigmaclosedQ}. 

We will enrich several coordinates of $s$ to obtain the conditions $q_{A'}$.
We mantain $A_{q_{A'}}=A_s$.
The main change will be adding the supremum of the $f_n$s which yields a bigger change in
$\mathcal F_s$ in order to prove  (5) of \ref{definitionQ}, this also implies the necessity of changing 
$\mathcal I_s$, and $\mathcal X_s$ and of course we will add the promise from \ref{mainextensionlemma} to
$\mathcal P_s$. The coordinate $\alpha_{q_{A'}}$  will become  $\alpha_s+\omega$.

\vskip 6pt
\noindent{\bf Claim 0:} Note that  $\tau_n\in \Sigma_{\eta_n, A_{p_n}} (A_s)$.
\vskip 6pt

\noindent{\bf Notation:} $\I$ will denote the set of all finite strictly
increasing sequences of the form $n_1< ...< n_k$ where $n_1, ..., n_k\in \N$ for
$k\in\N$. The empty sequence $\emptyset$ also belongs to $\I$.
\vskip 6pt
\noindent{\bf Claim 1:} Suppose $(n_1, ..., n_k)\in\I$.
If $x, y\in K$ are such that $\Pi \mathcal F_{s}(x)=\Pi \mathcal F_{s}(y)$,
then $\Pi \mathcal F_{p}\circ
\phi(\tau_{n_1})\circ ...\circ\phi(\tau_{n_k})(x)=\Pi \mathcal F_{p}\circ
\phi(\tau_{n_1})\circ ...\circ\phi(\tau_{n_k})(y).$\par

\noindent \underline{Proof of the claim:} 
Let $f\in \mathcal F_p\subseteq \mathcal F_{p_{n_1}}$, so
$f\circ \phi_{\tau_{n_1}}\in F_{p_{n_1+1}}\subseteq \mathcal F_{p_{n_2}}$.
Continuing this argument, we get $f\circ\phi(\tau_{n_1})\circ \phi(\tau_{n_2})\in
\mathcal F_{p_{n_3}}$, and so on, until $f\circ
\phi(\tau_{n_1})\circ ...\circ\phi(\tau_{n_k})\in \mathcal F_s$.
Now the claim follows directly from its hypothesis.

\vskip 6pt
The above claim justifies the introduction of the following:
\vskip 6pt
\noindent{\bf Notation:} 
If $(n_1, ..., n_k)\in \I$,  then  $f_n^{n_1,...,n_k}:\nabla F_{s}\rightarrow[0,1]$ is
a function satisfying: 
$$f_n^{n_1,...,n_k}\circ \Pi \mathcal F_{s}=f_n\circ \Pi \mathcal F_{p}\circ
\phi(\tau_{n_1})\circ ...\circ\phi(\tau_{n_k}).$$
\vskip 6pt
The continuity of the function above follows from the fact that
continuous surjections between compact spaces are quotient maps (see \cite{engelking}).
\vskip 6pt
\noindent{\bf Claim 2:} For every 
 $(n_1, ..., n_k)\in \I$ and every $n_{k+1}>n_k$,  we have 
 $$f_n^{n_1,...,n_k, n_{k+1}}\circ(\Pi \mathcal F_s)=f_n^{n_1,...,n_k}\circ (\Pi \mathcal F_s)\circ 
\phi(\tau_{n_{k+1}}).$$
\vskip 6pt
\noindent{\bf Claim 3:} 
 For every $(n_1, ..., n_k)\in \I$ the sequence
$(f_n^{n_1,...,n_k})_{n\in \N}$ is pairwise disjoint. \par
\noindent \underline{Proof of the claim:} Use \ref{preimages2-2disjoint} and the
fact that $\Pi\mathcal F_s$ is onto $\nabla \mathcal F_s$.
\vskip 6pt
\noindent{\bf Claim 4:}
For every $(n_1, ..., n_k)\in \I$ and every $m\in\N$ we have 
$f_n^{n_1,...,n_k}(x_{\xi_m}^s)=0$.\par

\noindent \underline{Proof of the claim:} We prove it by induction on $k\in\N$. 
For $k=0$, the claim is our assumption that 
$f(x^p_{\xi_m})=0$ and \ref{preimages2-2disjoint}. Suppose we have proved the claim up to $k$. 
Let $t_m\in K$ be such that $(\Pi\mathcal F_{s})(t_m)=x^s_{\xi_m}$.
We have
$$f_n^{n_1,...,n_k}(x_{\xi_m}^s)
=f_n^{n_1,...,n_k}\circ \Pi \mathcal F_{s}(t_m)=$$
$$=f_n\circ \Pi \mathcal F_{p}\circ
\phi(\tau_{n_1})\circ ...\circ\phi(\tau_{n_k})(t_m).$$

If $f\in \mathcal F_p$, then $f(t_m)=x^s_{\xi_m}(f)=x^{p_{n_1}}_{\xi_m}(f)=
x^{p_{n_1+1}}_{\xi_m}(f\circ \phi(\tau_{n_1}))=x^{p_{n_2}}_{\xi_m}(f\circ \phi(\tau_{n_1}))$
by +). Continuing this argument we get that
$$f(t_m)=x^{s}_{\xi_m}(f\circ \phi(\tau_{n_1})\circ ... \circ\phi(\tau_{n_k}))=
f\circ \phi(\tau_{n_1})\circ ... \circ\phi(\tau_{n_k})(t_m).$$
So, $$0=f_n(x^p_{\xi_m})=(f_n\circ \Pi \mathcal F_{p})(t_m)=
f_n\circ\Pi\mathcal F_p\circ \phi(\tau_{n_1})\circ ... \circ\phi(\tau_{n_k})(t_m)=$$
$$=f_n^{n_1,...,n_k}\circ \Pi\mathcal F_s(t_m)=f_n^{n_1,...,n_k}(x^s_{\xi_m}).$$
\vskip 6pt
\noindent{\bf Notation:} If $(n_1, ..., n_k)\in\I$, and $A\subseteq \N$
 then $f^{n_1, ..., n_k}_{A}$ is the supremum
of the pairwise disjoint sequence $(f^{n_1, ..., n_k}_{n}\circ \Pi\mathcal F_s)_{n\in A}$ in
$C_I(K)$.
 \vskip 6pt
\noindent{\bf Claim 5:} If $(n_1, ..., n_k)\in\I$, and $n_{k+1}>n_k$ and $A\subseteq \N$
 then, $f^{n_1, ..., n_k}_{A}\circ\phi(\tau_{n_{k+1}})=f^{n_1, ..., n_k, n_{k+1}}_{A}.$
\vskip 6pt
\noindent\underline{Proof of the claim:} This follows from \ref{supremacompositions} and claim 2.
\vskip 6pt
Now, applying \ref{strongalmostinfinite} and \ref{preservingpromisses} 
find an infinite $A\subseteq \N$ such that
whenever $A'$ is almost included in $A$, then if $M$ is an extension of
$\nabla F_{s}$ by $(f_n^{n_1,...,n_k})_{n\in A', (n_1, ..., n_k)\in \I}$, then
$M$ is a strong extension which preserves all the promisses $(L, R)\in \mathcal P_{q}$
in the sense of \ref{preservingpromisses}.
Moreover, we can assume, by choosing a convergent subsequence of $(x_{\xi_m}^s)_{m\in A}$
that $(x_{\xi_m}^s)_{m\in A}$ converges to a point which is also the limit of a convergent
sequence $(x_{\xi_m}^s)_{m\in B}$ for some $B\subseteq N\setminus A$.

\vskip 6pt
\noindent{\bf Notation:} If $A'$ is almost included in an $A$ as above, then we define
$$q_{A'}=(A_{s},\mathcal F_{s}\cup\{f^{n_1, ..., n_k}_{A'}: (n_1, ..., n_k)\in\I\}, \mathcal I, 
\alpha_{s}+\omega, \mathcal X_{q_{A'}}, \mathcal P_{s}\cup \{L',R'\}),$$
where 
\begin{enumerate}
\item $\mathcal I=\mathcal I_{q_{A'}}$ is the ideal of subsets of $A_s$ generated by the sets $A_{p_n}$ for $n\in\N$.
\item For each $\beta<\alpha_{q_{A'}}=\alpha_s$ the point
 $x^{q_{A'}}_\beta$ is the point of  $$GR(\Pi_{(n_1, ..., n_k)\in\I}\sum_{n\in A'}
 f^{n_1, ..., n_k}_{n})$$ above $x^{s}_\beta$.  
\item For $\beta\in [\alpha_s,\alpha_s+\omega)$ we choose the points
$x_\beta$ so that they form a  dense subset of $\nabla \mathcal F_{q_{A'}}$,
\item  $L'=\{\xi_m: m\in A'\}$ and $R'=\{\xi_m: m\in B\}$.
\end{enumerate}
\vskip 6pt
We will prove that $q_{A'}\in \Q$. First we will check that
 $(A_{s},\mathcal F_{q_{A'}}, \mathcal I)$ is a condition of $\PP$. Certainly 
(1)-(3) is true, and (4) follows from the fact that
$\mathcal F_s=\bigcup_{n\in \N}\mathcal F_{p_n}$ by \ref{sigmaclosedP},
the fact that $f_n$s depend on $A_{p_1}$ and by \ref{generalbijections}.
To check (5) we will prove the following:
\vskip 6pt
\noindent{\bf Claim 6:} Let  $\xi\in A_{s}$, $A\in \mathcal I$ and
let $f_1, ..., f_k$ for $k\in\N$ be 
functions in $\mathcal F_{q_A'}$ and $n\in\N$
such  that $A\subseteq A_{p_n}$, $\xi\in A_{p_n}$ and
$f_1, ..., f_k\in \mathcal F_{p_n}$. Let $m\in\N$ be such that $\Gamma_n(m)=\xi$. 
Let $n'\geq n$ such that $\Theta(n')=(\Theta_1(n'), \Theta_2(n'))=(n,m).$
Then $\tau_{n'}\in \Sigma_{\xi, A}(A_s)$ and 
$f_i\circ \phi(\tau_{n'})\in
F_{q_A'}$ for each $1\leq i\leq k$.\par

\noindent \underline{Proof of the claim:} By the choice of $\Theta$, there is $n'\geq n$ such that $$\Theta(n')=(\Theta_1(n'), \Theta_2(n'))=(n,m),$$
so $\eta_{n'}=\xi$. Since $p_{n'+1}$ is an amalgamation of two isomorphic conditions $p_{n'}$ and
$p_{n'}'$ and the isomorphism is witnessed by $\tau_{{n'}}$ satisfying
$\tau_{n'}\restriction \xi=Id_{\xi}$ and
$\tau_{n'}[A_{p_{n'}}\setminus\xi]>  \sup(A_{p_{n'}})\geq \sup(A)$, since $\eta_{n'}=\xi$, we have
$f_1\circ \phi(\tau_{n'}), ..., f_k\circ\phi(\tau_{n'})\in \mathcal F_{n'+1}\subseteq \mathcal F_s
\subseteq \mathcal F_{q_{A'}}$. This completes the proof of the claim.
\vskip 6pt
Now (5) of \ref{definitionP} follows from the above claim and Claim 5.
By the choice of $A$ based on \ref{strongalmostinfinite}
the space $\nabla \mathcal F_{q_{A'}}$ is a strong extension
of $\nabla F_s$ by $(f^{n_1, ..., n_k}_{n})_{{n\in A'}, (n_1, ..., n_k)\in \mathcal I}$ and so is
connected by Lemma \ref{strongconnected} since $\nabla \mathcal F_s$ was connected.
This gives 6) of the definition \ref{definitionP}. 

Now check Definition \ref{definitionQ}. We just checked (1), (2) is trivial,
(3) follows from the fact that the extension is strong and from the definition of $\mathcal X_{q_{A'}}$.
Now let us check (4) of \ref{definitionQ}.
If $(L,R)\in \mathcal P_s$, then $\{x^{q_{A'}}_\beta: \beta \in L\}
\cap \{x^{q_{A'}}_\beta: \beta \in R\}\not=\emptyset$ by the choice of $A\subseteq \N$
based on  \ref{preservingpromisses}.
Also 
$$\{x^{q_{A'}}_{\xi_m}: m \in A'\}
\cap \{x^{q_{A'}}_{\xi_m}: m \in B\}\not=\emptyset\leqno ++)$$ because
$\Pi_{(n_1, ..., n_k)\in\I}\sum_{n\in A'}
 f^{n_1, ..., n_k}_{n}(x^{q}_{\xi_m})=0$ for all $n,m\in \N$ by Claim 4 and so
$x^{q_{A'}}_{\xi_m}$ is $x^{q}_{\xi_m}$ followed by the sequence consisting only of zeros, hence
the choice of $B$ gives us ++) and (b) of the lemma.

Clearly $f^\emptyset$ is the indestructible supremum of $(f_n)_{n\in A'}$ which completes
the proof of the lemma.
\end{proof}

\section{Continuous functions in $V$ and $V[G]$}

In this and in the following section we will employ the partial order $\Q$ as
a forcing notion (see \cite{kunen}, \cite{jech}). We start with a set-theoretic universe
$V$ which satisfies the CH and we will consider its generic extension $V[G]$
where $G$ is a $\Q$-generic over $V$.

By  \ref{sigmaclosedQ} and \ref{omega2ccQ}, forcing with $\Q$
does not collapse cardinals  if CH holds in $V$ and does not add any new countable
subsets of the universe (see [Ku]),
i.e., the ground set-theoretic universe $V$ and 
the generic extension $V[G]$ have the
same countable sets of $V$ and the same cardinals.\par
This, in particular, means that the completion $Co(\omega_2)$ of the free
Boolean algebra $Fr(\omega_2)$ with $\omega_2$-generators is the same
in both of the universes $V$ and $V[G]$.
The Stone space $K^*$ of $Co(\omega_2)$ is bigger in $V[G]$, but $K$
is dense in it, so every $f\in C(K)$ as a uniformly continuous function,
uniquely extends to a function $f^*\in C(K^*)$. If $\mathcal F\subseteq C_I(K)$,
then $\mathcal F^*$ will denote $\{f^*: f\in \mathcal F\}$.

\noindent As $G$ is a $\Q$-generic filter,  for every two $p, p'\in G$
 there is $q\leq p, p'$ with $q\in G$
and $G\cap D\not=\emptyset$ for any dense set
$D\subseteq \Q$ in the ground set-theoretic universe $V$. 
In the generic extension $V[G]$ we consider $\mathcal F_G=\bigcup \{\mathcal F_p: p\in G\}$
and the 
compact space  $L=\nabla \mathcal F^*_G$
 with the subset $\mathcal X=\{x_\beta: \beta<\omega_1\}$ such that
$x_\beta(f^*)=x_\beta^p(f)$ for any $p\in G$ and $\beta<\alpha^p$.

\begin{lemma}\label{fstarlemma}
Suppose that $p\forces \dot f\in C_I(K^*)$, then there is 
$f\in C_I(K)$ and $q\leq p$ such that $q\forces \check f^*=\dot f$.
\end{lemma}
\begin{proof}
In $V[G]$ there is a sequence of $f_n^*=\sum_{i\in\N} r_{in}\chi_{[a_{in}]}^*$ such that
$f_n$ uniformly converges to $f$ where $r_{in}$s are reals and 
$a_{in}$s are elements of $Co(\omega_2)$.  Both the reals and the elements of
$Co(\omega_2)$  are the same in $V$ and $V[G]$ and so building a decreasing
sequence of conditions of $\Q$ we can successively decide them. Find the lower
bound $q$ of such a sequence, obtained by \ref{sigmaclosedQ} and
define $f$ to be the uniform limit of
 $f_n=\sum_{i\in\N} r_{in}\chi_{[a_{in}]}$s (which must exist since the sequence must be
a Cauchy sequence as it converges in $V[G]$). Two continuous functions agreeing 
on a dense set must be equal, so $q\forces \check f^*=\dot f$.
\end{proof}

\begin{lemma} Suppose that $\mathcal F\subseteq C_I(K)$ is countable,
then $\nabla \mathcal F=\nabla \mathcal F^*$, and so
$C_I(\nabla \mathcal F)=C_I(\nabla \mathcal F^*)$.
\end{lemma}
\begin{proof} Let $\mathcal A$ be a countable subalgebra  of $Co(Fr(\omega_2))$ such that
$f$ depends on $\mathcal A$ for every $f\in\mathcal F$. Let $u\in K^*$ be such that
$(\Pi \mathcal F^*)(u)=t\in [0,1]^{\mathcal F^*}$. $u\cap \mathcal A$  belongs
to $V$ as it is a countable subset of $V$. Now extend $u$ to any ultrafilter
$v$ of $Co(Fr(\omega_2))$ which is in $V$, we have $(\Pi\mathcal F)(v)(f)=(\Pi\mathcal F^*)(u)(f^*)$.
\end{proof}

\begin{lemma}\label{decidingfunctions}
 Suppose  that $p\in\Q$ and $p\forces \dot f\in C_I(\nabla \mathcal F_G^*)$.
There is $q\leq p$ in $\Q$ and $g\in C_I(\nabla\mathcal F_q)$ such that
$$q\forces \check g\circ\pi_{\mathcal F_q^*, \mathcal F_G^*}=\dot f$$.
\end{lemma}
\begin{proof}
As in \ref{countabledependence} one can prove that $f$ depends on countably many
coordinates in $\mathcal F_G^*$.  One can decide this countable set. Using the compatibility
of all elements in the generic $G$ and the fact that $\Q$ is $\sigma$-closed one can
built $q'\leq p$ such  that 
$$q'\forces \dot f= \dot g\circ \pi_{F_{q'}^*, \mathcal F_G^*},\ \ \  \dot g\in C_I(\nabla\mathcal F_{q'}^*).$$
So now find $q\leq q'$ which decides $\dot g$ as $g$ using the previous lemma.
\end{proof}

\section{The construction and the properties of the  space}

\begin{lemma}\label{finallemma} 
The compact space $L$ with its subset $\mathcal X$
 has the following properties in $V[G]$:
\begin{enumerate}[(A)]
\item The weight 
of $L$ is $\omega_2$.
\item $\mathcal X=\{x_\beta:\beta<\omega_1\}$ is a dense subset of $L$
\item  Given
\begin{enumerate}
\item a sequence of  pairwise disjoint 
elements $(f_n:n\in N)$ of $C_I(L)$,\par
\item a sequence $(\xi_m:m\in N)\subseteq\omega_1$
such that $f_n(x_{\xi_m})=0$ for all $n,m\in N$,
and $\{x_{\xi_m}: m\in N\}$ is relatively discrete, \par

\end{enumerate}
  there  is an
 infinite  $A'\subseteq \N$  
such that 
\begin{enumerate}[(1)]
\item The supremum
$f=sup\{f_n:\ n\in A'\}$ exists in $C_I(L)$,
\item $\overline{\{x_{\xi_m}: m\in A'\}}\cap \overline{\{x_{\xi_m}: m\not\in A'\}}\not=\emptyset$
\end{enumerate}

\item If $U_1$, $U_2$ are two
open subsets of $L$
then ${\overline {U_1}}\cap {\overline{U_2}}=\emptyset$ or 
${\overline {U_1}}\cap {\overline{U_2}}$ has 
at least two points.
\end{enumerate}
\end{lemma}

\noindent PROOF: 
\noindent (A) follows from \ref{nablaofdealphas}, \ref{definitionP} (2) and the fact that
$\bigcup\{A_p: p\in G\}$  contains unbounded in $\omega_2$ set of limit ordinals
by the standard density argument, which can be easily
obtained from \ref{amalgamationQ}. (B) follows from the standard density arguments since 
we can increase $\alpha_p$s arbitrarily in $\omega_1$.

\noindent To get (C) work in $V$ and
fix $p\in \Q$. Let ${\dot f_n}$'s, ${\dot x}_{\xi_m}$'s,
 be $\Q$-names for the objects mentioned
in items a) - b) of (C). We will produce
$q\leq p$ which will force (1) - (2) of the lemma.
By \ref{decidingfunctions} and \ref{sigmaclosedQ} there is $q'\leq p$
and $\xi_m$s 
and functions $g_n:\nabla F_{q'}\rightarrow [0,1]$ such that 
$\xi_m<\alpha_{q'}$, $q'\forces \check \xi_m=\dot \xi_m$, and such that
$$q'\forces \check g_n\circ \pi_{\mathcal F_{q'}^*, \mathcal F_G^*}=\dot f_n.$$
It follows that $g_n$s are pairwise disjoint and that $g_n(x_{\xi_m}^{q'})=0$
for all $m,n\in \N$.

Now use \ref{mainextensionlemma} to find an extension $q\leq q'$ where we have an
indestructible supremum of $g_n$'s for $n\in A'$ and a promise that
 $\overline{\{x_{\xi_m}: m\in A'\}}\cap \overline{\{x_{\xi_m}: m\not\in A'\}}\not=\emptyset.$
The supremum and the above nonempty intersection of
the closures remain in $\nabla \mathcal F_G^*=L$ because if they failed this would be
witnessed by some countable $\mathcal F\subseteq \mathcal F_G^*$ (as continuous functions
depend on countably many coordinates and separations of closed sets can be 
obtained by two unions with disjoint closures of
 finitely many basic open sets) and we could
use the $\sigma$-closure of $\Q$ and the compatibility of conditions in the
generic $G$ to obtain $s\leq q$ such that the supremum or the condition about
the closures are destroyed in $\nabla \mathcal F_s$ which is impossible
since the supremum is indestructible and the promise preserved by  the definition
\ref{definitionQ} of the order in $\Q$.

To prove (D), first note that $L$ is c.c.c. as a continuous image of a c.c.c. space
$K=Co(\omega_2)$, so we may assume that $U_1$ and $U_2$ are unions of countably many
basic sets. Deciding them all and using \ref{sigmaclosedP} we may assume that they are all
determined by coordinates of $\nabla \mathcal F_p$ for some $p\in\Q$. I.e. that we have $V_1, V_2$
open subsets of $\nabla \mathcal F_p$ such that $U_i=\pi_{\mathcal F_p,\dot{\mathcal F}}^{-1}[V_i]$ for $i=1,2$,
where $\Q$ forces that $\dot{\mathcal F}=\{\mathcal F_p: p\in G\}$.
Of course we must have $\overline V_1\cap\overline V_2\not=\emptyset$.
 Now find an $A\subseteq \omega_2$
such that $A_p<A$ and there is an order preserving bijection $\tau: A_p\rightarrow A$.
Transport all the structure of $p$ to $A$ by $\tau$ obtaining a condition $p'\in \Q$ such that
$\tau$ witnesses that $p$ and $p'$ are isomorphic. Amalgamate them according to
\ref{amalgamationP} obtaining $q\leq p,p'$. 
Using \ref{disjointultrafilters} it is easy to prove that 
$$\nabla F_q=(\nabla F_p)\times (\nabla F_{p'}).$$
So taking two distinct points $x,y\in \nabla F_{p'}$ we have that
$(t, x)\in \overline  V_1 \times \{x\}\cap \overline V_2 \times \{x\}\subseteq \overline U_1\cap \overline U_2$ and 
$(t, y)\in \overline V_1 \times \{y\}\cap \overline V_2 \times \{y\} \subseteq \overline U_1\cap \overline U_2$ for any
$t\in \overline V_1\cap\overline V_2$. So, $(t,x), (t,y)\in \overline U_1\cap\overline U_2$
completing the proof of (C).

\begin{theorem}\label{finaltheorem} $L$  is a compact space such that
\begin{enumerate}
\item The density of $C(K)$ is $2^{2^\omega}>2^\omega$.
\item Every linear bounded operator $T: C(L)\rightarrow C(L)$ is of the form
$T(f)=gf+S(f)$ where $g\in C(K)$ and $S$ is weakly compact (strictly singular)
\item $C(K)$ is an indecomposable Banach space, in particular it has no
complemented subspaces of density less or equal to  $2^\omega$,
\item $C(K)$ is not isomorphic to any of its proper subspaces nor any
of its proper quotients.
\end{enumerate}
\end{theorem}
\begin{proof}
(1) follows from (A) of \ref{finallemma}, a standard counting argument in a generic extension (see \cite{kunen})
and the fact that  $\Q$ does not add new reals to a model of CH.
 (2) is proved the following way: first we prove that every
operator on $C(L)$ is a weak multiplier (see of 2.1 and 2.2. of  \cite{few}) exactly the same way as
Lemma 5.2 of \cite{few}, then we use the fact that for every $x\in L$ the space
$L\setminus\{x\}$ is $C^*$-embedded in $L$ which follows from C) of \ref{finallemma}
and Lemma 2.8. of \cite{few}. Now Lemma 2.7 of \cite{few} implies (2) of our theorem.
(3) follows, for example, from Lemma 3.4. of \cite{rogerio}.
(4) follows from 2.3 of \cite{few} and the fact that of course, $L$ cannot have a
nontrivial  convergent sequence
which would give a complemented copy of $c_0$ contradicting (3).
\end{proof}

\bibliographystyle{amsplain}

\end{document}